\documentclass[11pt, reqno, a4paper]{amsart}
\usepackage{graphicx}
\usepackage{float}
\usepackage{tikz}
\usepackage{etoolbox}

\usepackage[a4paper,margin=1in]{geometry}
\usepackage{amsmath,amssymb,amsthm,mathtools}
\usepackage{enumitem}
\usepackage{hyperref}

\newtheorem{theorem}{Theorem}[section]
\newtheorem{corollary}[theorem]{Corollary}
\newtheorem{lemma}[theorem]{Lemma}
\newtheorem{proposition}[theorem]{Proposition}
\theoremstyle{definition}
\newtheorem{definition}[theorem]{Definition}
\newtheorem{remark}[theorem]{Remark}
\newtheorem{example}[theorem]{Example}
\theoremstyle{plain}
\newtheorem{claim}{Claim}
\AtBeginEnvironment{definition}{\vspace{0.2em}}
\AtEndEnvironment{definition}{\vspace{0.2em}}
\AtBeginEnvironment{remark}{\vspace{0.2em}}
\AtEndEnvironment{remark}{\vspace{0.2em}}
\AtBeginEnvironment{lemma}{\vspace{0.2em}}
\AtEndEnvironment{lemma}
{\vspace{0.2em}}
\AtBeginEnvironment{theorem}{\vspace{0.2em}}
\AtEndEnvironment{theorem}{\vspace{0.2em}}
\newcommand{\dom}{\text{\rm dom}}
\newcommand{\diam}{\text{\rm diam}}

\title{Borel Graphs Generated by Commuting Functions}
\author{Su Gao}
\address{School of Mathematical Sciences and LPMC, Nankai University, Tianjin 300071, P.R. China}
\email{sgao@nankai.edu.cn}
\thanks{The authors acknowledge the partial support of their research by the Fundamental Research Funds for the Central Universities and by the National Natural Science Foundation of China (NSFC) grant 12271263. The paper was completed when all of the authors were attending a Tianyuan Workshop on Definability and Computation; the authors thank the Tianyuan Mathematics Research Center for their partial support of the research.}

\author{Xiangxi Hu}
\address{School of Mathematical Sciences and LPMC, Nankai University, Tianjin 300071, P.R. China}
\email{xxhu@mail.nankai.edu.cn}

\author{Jie Zou}
\address{School of Mathematical Sciences and LPMC, Nankai University, Tianjin 300071, P.R. China}
\email{jzou@mail.nankai.edu.cn}


\begin{document}
\maketitle
\begin{abstract}
In this paper we study Borel graphs generated by finitely many commuting Borel functions. We give a geometric analysis of the free part of such graphs based on marker sets and marker regions.  Assuming the existence of 
$r$-forward-independent hitting sets with bounded syndeticity, we obtain marker decompositions of 
the free part into rootless and rooted regions with controlled geometry. As
applications, we derive finite Borel asymptotic dimension and hyperfiniteness, and obtain upper bounds for Borel edge
chromatic numbers which improve previously known results. For the case in which each of the commuting Borel functions is  bounded-to-one, we verify the existence of $r$-forward-independent hitting sets with syndeticity $Cr$ for some constant $C$. This gives another proof of a recent theorem of Shinko--Weilacher--Yu, and is used to show that if one of the commuting Borel functions is injective and another one is bounded-to-one and exactly even-to-one, then the graph has a Borel perfect matching.
\end{abstract}
%-----------------------
\section{Introduction}

Descriptive combinatorics studies classical combinatorial problems under constraints of definability, with a particular focus on Borel graphs whose connectedness relations are countable Borel equivalence relations. A central theme in this study is to understand how additional structure in a Borel graph can
be used to obtain Borel colorings, Borel edge colorings, Borel matchings, and
hyperfiniteness. 

A natural class of Borel graphs consists of graphs generated
by finitely many Borel functions. If $X$ is a Polish space or a standard Borel space, and $f_1, \dots, f_n\colon X \to X$ are Borel unary functions, then the graph $G_{f_1,\dots, f_n}$ is defined as $(X, R)$, where
$$ \{x, y\}\in R\iff x\neq y \mbox{ and } \exists 1\leq i\leq n\ [\, f_i(x)=y \mbox{ or } f_i(y)=x\, ]. $$
A systematic study of such graphs started with the seminal paper by Kechris--Solecki--Todorcevic \cite{KechrisSoleckiTodorcevic1999}. The following results about the Borel chromatic numbers were shown in \cite{KechrisSoleckiTodorcevic1999}:
\begin{itemize}
\item When $n=1$, the Borel chromatic number of $G_f$ is either $1, 2, 3$ or $\aleph_0$;
\item In general, if $G_{f_1,\dots, f_n}$ has finite Borel chromatic number, then it is at most $3^n$;
\item If each $f_1, \dots, f_n$ is $\leq\!k$-to-$1$, then the Borel chromatic number of $G_{f_1,\dots, f_n}$ is at most
$$ \min\{(k+1)n+1, \, 3^n\}. $$
\end{itemize}
They asked if the last bound can be improved to $2n+1$; this problem is still open. 

When the generating functions commute with each other, 
the resulting graph has additional structure and often better results can be proved than in the general noncommuting setting. For example, it was shown in \cite{Palamourdas2012,MeehanPalamourdas2021} that, if $f_1, \dots, f_n$ are commuting functions with no fixed points, the Borel chromatic number of $G_{f_1,\dots, f_n}$ is indeed at most $2n+1$ if it is finite. In contrast, for general noncommuting functions, a known bound not involving $k$ is only $\frac{1}{2}(n+1)(n+2)-2$.

In this paper, we study Borel edge colorings, Borel matchings, and hyperfiniteness of Borel graphs generated by finitely many commuting Borel functions. A general result from \cite{KechrisSoleckiTodorcevic1999} implies that, for any $\leq\!k$-to-$1$ functions $f_1, \dots, f_n$ (not necessarily commuting), the Borel edge chromatic number of $G_{f_1,\dots, f_n}$ is at most $2n(k+1)-1$.  With commuting assumptions, we improve the bound to $(n+1)(k+3)-5$.  We also prove under additional assumptions that Borel perfect matchings exist and that the connectedness relation of the graph is hyperfinite. 

In the special case that the generating functions $f_1,\dots, f_n$ are Borel isomorphisms, the resulting graph $G_{f_1,\dots, f_n}$ is the Schreier graph of a Borel action of the group $\langle f_1, \dots, f_n\rangle$. Furthermore, if $f_1,\dots, f_n$ are commuting with each other, we have an action of a countable abelian group. For countable abelian group actions, the Borel chromatic numbers, Borel edge chromatic numbers, Borel matchings, and their hyperfiniteness have all been understood (see \cite{JacksonKechrisLouveau2002, GaoJackson2015, GaoJacksonKrohneSeward2024, ConleyJacksonMarksSewardTuckerDrob2020, Bencs2021FactorIID, GrebikRozhon2023LocalProblems, Weilacher2024, BorelFactorsSubshifts2025}). Our approach in this paper is inspired by \cite{GaoJackson2015} and the more recent \cite{GaoJacksonKrohneSeward2024}, in which an elaborate theory of marker structures was developed for countable abelian group actions. 

In this paper our main effort is to develop a geometric analysis of the Borel graph $G_{f_1,\dots, f_n}$ based on marker sets and marker regions. For this purpose, it is convenient to view $G_{f_1,\dots, f_n}$ as consisting of orbits of a monoid action, and we will work on the free part of this action, namely 
$$
\begin{aligned}
F(X)=
\bigl\{x\in X:\ &
\forall \alpha_1,\ldots, \alpha_n, \beta_1,\ldots, \beta_n\in\mathbb N,\\
&
(f_1^{\alpha_1}\circ\cdots\circ f_n^{\alpha_n})(x)
=
(f_1^{\beta_1}\circ\cdots\circ f_n^{\beta_n})(x)
\Longrightarrow
\alpha_i=\beta_i
\text{ for every }1\le i\le n
\bigr\}.
\end{aligned}
$$
% More precisely, we work with graphs of the form $G_{f_1,\ldots,f_n}\upharpoonright F(X)$, where $f_1,\ldots,f_n$ are commuting Borel functions and $F(X)$ denotes the free part. We often assume that the generating functions are countable-to-one, so that the associated connected equivalence relation is countable Borel. For applications requiring bounded degree, such as edge colorings, we impose stronger $k$-to-one assumptions.

Figure~\ref{Fig.1} illustrates the local geometry of
$G_{f_1,f_2}\upharpoonright F(X)$ in a typical case, where $f_1$ is
injective and $f_2$ is $2$-to-$1$.

\begin{figure}[H]
  \centering
  \includegraphics[width=0.38\textwidth]{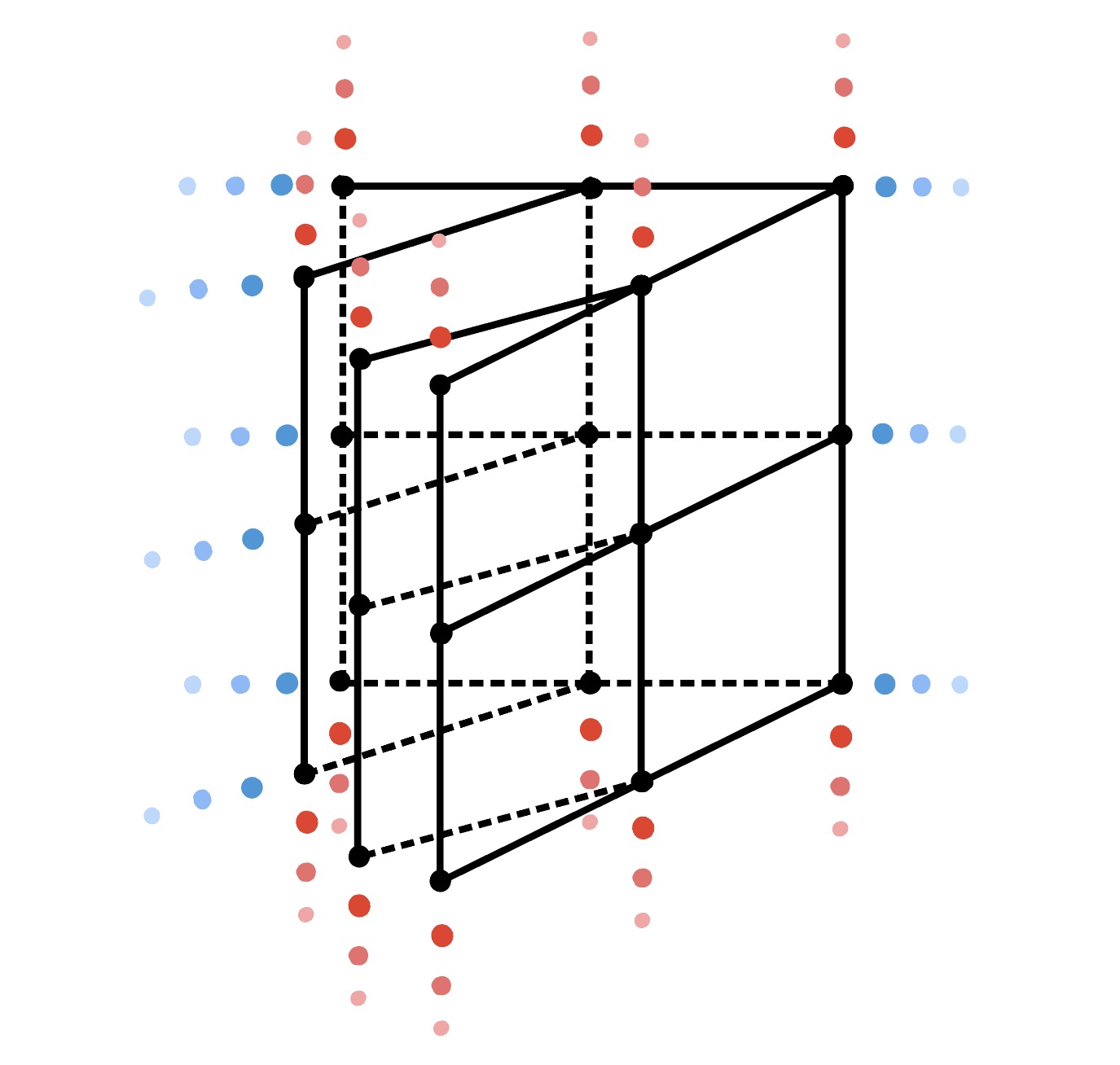}
  \caption{The free part of the graph.}
  \label{Fig.1}
\end{figure}

An important technique to prove hyperfiniteness is to show finite Borel asymptotic
dimension (see \cite{ConleyJacksonMarksSewardTuckerDrob2023}). However, finite Borel asymptotic dimension does not necessarily hold even when there is only one generating function. For example, consider the shift map $S:[\mathbb N]^\mathbb{N}\to [\mathbb N]^\mathbb{N}$, $S(x)(m)=x(m+1)$, where $[\mathbb N]^\mathbb{N}=\{x\in \mathbb{N}^\mathbb{N}:\forall n\in \mathbb{N},\ x(n)<x(n+1)\}$. The graph $G_S$ generated by this shift map does not have finite Borel
asymptotic dimension. This follows from the fact that $G_S$ has infinite Borel chromatic number (\cite[Example 3.2]{KechrisSoleckiTodorcevic1999}) and \cite[Corollary 8.2]{ConleyJacksonMarksSewardTuckerDrob2023}. On the other hand, for any Borel function $f$, the connectedness relation of $G_f$ is hyperfinite (\cite[Corollary 8.2]{DoughertyJacksonKechris1994}). 

Since we consider countable-to-one Borel functions in this paper, which generate locally countable Borel graphs, another issue is that finite Borel asymptotic dimension does not imply hyperfiniteness in this context. For this we define a stronger notion of finite strong Borel asymptotic dimension, which does imply hyperfiniteness for locally countable Borel graphs, and use it throughout the paper.

In \cite{GrebikHiggins2026}, finite Borel asymptotic dimension for graphs generated by
a single Borel function was characterized in terms of forward-independent
hitting sets (\cite[Theorem 1.2]{GrebikHiggins2026}). More precisely, for such a graph $G_f$, finite Borel asymptotic
dimension is equivalent to the existence, for every $r\in\mathbb N^+$, of a
Borel $r$-forward-independent hitting set. 
Motivated by this result, we define a notion of $r$-forward-independent hitting set with bounded syndeticity for $G_{f_1,\dots, f_n}$ and prove its sufficiency for finite strong
Borel asymptotic dimension. 

The basic result in \cite{GaoJackson2015} is a Borel marker decomposition of the phase space of a free Borel $\mathbb Z^n$-action into $n$-dimensional rectangular regions with edge lengths $d$ or $d+1$ (\cite[Theorem~3.1]{GaoJackson2015}). This was the starting point of a hyperfiniteness proof for a broad class of shift actions (\cite[Theorem~8.1]{GaoJackson2015}).
The present setting is different from the group-action setting because the
generating maps need not be invertible. This creates an asymmetric geometry:
the functions can be always applied in the forward direction, while in the backward direction there may be many
preimages or no preimage at all. For this reason, the rectangular marker
constructions for $\mathbb Z^n$-actions cannot be applied directly. In this paper, we develop some structural decomposition results for 
graphs generated by finitely many commuting Borel functions. The main output is
a family of Borel marker regions on the free part $F(X)$, whose shapes are
controlled in all generating directions.

We state the main consequences of this decomposition below. The first result is
the structural tiling theorem, which is proved in Section~4.
\begin{theorem}
Let $X$ be a standard Borel space, and let $d\in \mathbb{N}^+$. For $n\in \mathbb{N}$, suppose that $f_1,\cdots ,f_n:X\to X$ are countable-to-one commuting Borel functions. Assume that for every $r\in\mathbb N^+$, there is a Borel $r$-forward-independent
hitting set with syndeticity $Cr$ for
$G_{f_1,\ldots,f_n}\upharpoonright F(X)$,
where $C\in\mathbb N^+$ is independent of $r$. Then there exists a smooth Borel subequivalence relation
$R_d\subseteq E_{f_1,\cdots,f_n}\upharpoonright F(X)$
such that every $R_d$-class $A$ is a rooted region satisfying
 $d_i(A)=d$ or $d_i(A)=d+1$ for $1\le i \le n$. 
\end{theorem}

The technical hypothesis in the above theorem is necessary. In fact, for each $d$, all of the roots of the rooted regions in the above theorem form a $d$-forward-independent hitting set with syndeticity $nd$. 

The rooted marker region decomposition has an immediate application to Borel
asymptotic dimension and hyperfiniteness.
The next theorem gives a sufficient condition for finite Borel asymptotic
dimension, and hence for hyperfiniteness, for graphs generated by commuting
Borel functions.

\begin{theorem}
Let $X$ be a standard Borel space. For $n\in\mathbb{N}$, suppose that $f_1,\cdots,f_n:X\to X$ are countable-to-one commuting Borel functions, and let $\rho$ be the graph metric on $G_{f_1,\cdots,f_n}\upharpoonright F(X)$. Assume that for every $r\in\mathbb N^+$, there is a Borel $r$-forward-independent hitting set with syndeticity $Cr$ for $G_{f_1,\ldots,f_n}\upharpoonright F(X)$, where $C\in\mathbb N^+$ is independent of $r$. 
Then $(F(X),\rho)$ has finite strong Borel asymptotic dimension. In particular,
$E_{f_1,\cdots,f_n}\upharpoonright F(X)$ is hyperfinite.
\end{theorem}

Recently it has been shown by Shinko--Weilacher--Yu \cite{NaryshkinShinkoWeilacherYuCommutingFunctions} that if $\{f_i\}_{i\in I}$ is a countable family of commuting Borel functions each of which is bounded-to-one, then $E_{\{f_i\}_{i\in I}}$ is hyperfinite. In fact, they showed that if $f_1, \dots, f_n$ are bounded-to-one commuting Borel functions, then the Borel asymptotic dimension of the free part of $G_{f_1,\dots, f_n}$ is indeed finite. In this context, we have the following result.

\begin{theorem} Let $X$ be a standard Borel space, and let $f_1, \dots, f_n\colon X\to X$ be bounded-to-one commuting Borel functions. Then for any $r\in \mathbb{N}^+$, there is a Borel $r$-forward-independent hitting set with syndeticity $2n^2r$ for $G_{f_1,\dots, f_n}\upharpoonright F(X)$.
\end{theorem}

By combining the above two theorems, we yield another proof of the theorem of Shinko--Weilacher--Yu \cite{NaryshkinShinkoWeilacherYuCommutingFunctions} for $G_{f_1,\dots, f_n}\upharpoonright F(X)$.

The same decomposition method also gives applications to Borel perfect
matchings. This is similar to the method used in the recent work \cite{GaoJacksonKrohneSeward2024},
where it was shown that free Borel actions of $\mathbb Z^n$ admit Borel perfect
matchings for the associated Schreier graphs (\cite[Theorem~6.1]{GaoJacksonKrohneSeward2024}).
In the present noninvertible setting, we obtain the following analogue under
suitable assumptions on the generating functions.

\begin{theorem}
Let $X$ be a standard Borel space and let $I$ be an index set.
Suppose that $\{f_i\}_{i\in I}$ is a family of Borel functions on $X$ such that $f_1\circ f_2=f_2\circ f_1$, $f_1$ is injective and $f_2$ is bounded-to-one and exactly even-to-one.
Then, there exists a Borel perfect matching for $G_{\{f_i\}_{i\in I}}\upharpoonright F(X)$.
\end{theorem}

It has been shown in \cite[Proposition 2.8]{ConleyMiller2017MeasurablePerfectMatchings} that for any countable-to-one Borel surjection $f$ there exists a Borel perfect matching for $G_f$ on the non-injective part of $f$. In particular, if $\{f_i\}_{i\in I}$ is a family of Borel functions on $X$ with at least one of $f_i$ being a countable-to-one surjection, then there exists a Borel perfect matching for $G_{\{f_i\}_{i\in I}}$ on the non-injective part for this $f_i$. Our theorem above does not assume surjectivity.

Now we can state our result on the Borel edge chromatic number of $G_{f_1,\dots, f_n}$ as an application of our method. 
For Schreier graphs of
free shifts of $\mathbb Z^n$, it was known that the Borel edge chromatic number
coincides with the classical edge chromatic number:
$$
\chi'_B(F(2^{\mathbb Z^n}))=\chi'(F(2^{\mathbb Z^n}))=2n.
$$
This was proved independently in
\cite{Bencs2021FactorIID, GrebikRozhon2023LocalProblems, Weilacher2024, BorelFactorsSubshifts2025}. In fact, these results imply that there are Borel perfect matchings for $F(2^{\mathbb{Z}^n})$. 
For graphs generated by noninvertible functions, however, the lack of inverses
creates difficulties. Using the marker decompositions developed in
this paper, we prove the following upper bound.

\begin{theorem}
    Let $X$ be a standard Borel space. For $n\in\mathbb{N}$, suppose that $f_1,\cdots,f_n:X\to X$ are $k$-to-1 commuting Borel functions.
 Then $\chi'_B(G_{f_1,...,f_n}\upharpoonright F(X))\leq (n+1)(k+3)-5.$
\end{theorem}

\medskip
\noindent\textbf{Organization.}
The rest of the paper is organized as follows. Section~2 collects the necessary background
from descriptive combinatorics and Borel asymptotic dimension. In Section~3 we prove the existence of $r$-forward-independent hitting sets in the context of bounded-to-one commuting Borel functions. Section~4 proves
the structural decomposition theorems. Section~5 presents applications to
finite Borel asymptotic dimension, hyperfiniteness, Borel perfect matchings,
and Borel edge colorings.

\medskip
\noindent\textbf{Acknowledgments.} We thank Forte Shinko and Felix Weilacher for helpful discussions on the topic of the paper. Specifically, they pointed out a mistake in an earlier version of this paper and suggested Definition~\ref{def:fBasd} and Lemma~\ref{lem:hyp}. Weilacher also suggested all materials in Section 3.

%---------------------------
\section{Preliminaries}

In this section we collect several standard tools from descriptive combinatorics that will be used throughout the paper. 

\subsection{Basics in descriptive set theory}

We begin with some standard terminology from descriptive set theory. A standard reference is \cite{Kechris1995}.

A topological space $X$ is \emph{Polish} if it is separable and completely metrizable.
A \emph{standard Borel space} is a measurable space $(X,\mathcal{B})$ such that $\mathcal{B}$ is the $\sigma$-algebra generated by some Polish topology on $X$, and sets in $\mathcal{B}$ are called \emph{Borel sets}.

A function on a standard Borel space is said to be a \emph{Borel function} if the preimage of every Borel set is also Borel.
Let $X$ be a standard Borel space. A Borel function $f\colon X\to X$ is \emph{countable-to-one} (respectively, \emph{finite-to-one}) if for each $x\in X$, $f^{-1}(x)$ is countable (respectively, finite). Moreover, a finite-to-one function $f\colon X\to X$ is \emph{bounded-to-one} if there is $N\in \mathbb{N}^+$ such that $|f^{-1}(x)|\leq N$ for all $x\in X$. We work with the following refinements of the notion of bounded-to-one functions.
\begin{definition}
   Let $X$ be a standard Borel space, $k\in \mathbb{N}^+$, and let $f: X\to X$. We say that $f$ is \emph{$k$-to-1} (or \emph{$\leq\!k$-to-1} as in \cite{KechrisSoleckiTodorcevic1999}) if $|f^{-1}(x)|\le k$ for every $x\in X$. 
   We say that $f$ is \emph{exactly $k$-to-1} if $|f^{-1}(x)|= k$ for every $x\in X$. 
   We say that $f$ is \emph{exactly even-to-one} if $|f^{-1}(x)|$ is a finite even number for every $x\in X$.
\end{definition}

%There is a theorem in \cite{Kechris1995} as follows, showing that if each section of  a Borel set is countable,  then it can be written as a countable union of Borel functions.
%\begin{lemma}\label{lem:uniform}
%   Let $X,Y$ be standard Borel spaces and let $P\subseteq X\times Y$ be Borel. If every section $P_x$ is countable, then $P$ has a Borel uniformization and therefore $\text{proj}_X(P)$ is Borel.

%    Moreover, $P$ can be written as $\bigcup_n P_n$, where each $P_n$ is a Borel graph.
%\end{lemma}

An equivalence relation $E$ on a standard Borel space $X$ is called a \emph{countable} (respectively, \emph{finite}) \emph{Borel equivalence relation} if $E$ is a Borel subset of $X\times X$ and each equivalence class is countable (respectively, finite). A countable Borel equivalence relation is said to be \emph{hyperfinite} if it can be written as an increasing union of finite Borel equivalence relations. 

Let $E$ be a countable Borel equivalence relation on a standard Borel space $X$. A \emph{selector} for $E$ is a function $s\colon X\to X$ such that for any $(x,y)\in E$, $s(x)=s(y)Ex$. We say that $E$ is \emph{smooth} if there is a Borel selector for $E$. Any finite Borel equivalence relation is smooth.

An equivalence relation $E$ is \emph{hypersmooth} if $E$ can be written as an increasing union of smooth Borel equivalence relation. Any hypersmooth countable Borel equivalence relation is hyperfinite (\cite[Theorem 5.1 (3)]{DoughertyJacksonKechris1994}).

\subsection{Borel graphs and combinatorial notions}
We recall some standard notions from Borel graph combinatorics. A standard reference is \cite{KechrisMarks2020}.

Let $X$ be a standard Borel space, and let $G=(X,R)$ be a graph on $X$.
We say that $G$ is a \emph{Borel graph} if $R\subseteq X\times X$ is a Borel, symmetric, and irreflexive relation. A Borel graph $G$ is \emph{locally countable} (respectively, \emph{locally finite}) if the degree of every vertex is contable (respectively, finite). 

For any graph $G$ we let $\chi(G)$ denote the chromatic number of $G$ and $\chi'(G)$ denote the edge chromatic number of $G$. The following definitions are their Borel counterparts.

\begin{definition}
Let $G=(X,R)$ be a Borel graph. 
\begin{enumerate}
\item A \emph{ proper Borel coloring} of $G$ is a Borel function $c\colon X\to Y$ with $Y$ a Polish space such that 
$$
xRy \ \Longrightarrow\ c(x)\neq c(y)\qquad \text{for }x,y\in X.
$$
The minimum of $\{|Y|\colon \mbox{there is a proper coloring $c\colon X\to Y$} \}$  is called the \emph{Borel chromatic number} of $G$, denoted by $\chi_B(G)$.
\item A \emph{proper Borel edge coloring} of $G$ is a
Borel function $c:R\to Y$  with $Y$ a Polish space such that 
$$
e_1\cap e_2\neq\varnothing \ \Longrightarrow\ c(e_1)\neq c(e_2)\qquad \text{for }e_1,e_2\in R.
$$
The minimum of $\{|Y|\colon \mbox{there is a proper edge coloring $c\colon R\to Y$}\}$ is called the \emph{Borel edge chromatic number} of $G$, denoted by $\chi'_B(G)$.
\end{enumerate}
\end{definition}

Any locally countable Borel graph has countable Borel edge chromatic number (\cite[Proposition 4.10]{KechrisSoleckiTodorcevic1999}).

Another central notion in classical combinatorics is perfect matching. We shall use the following Borel version.

\begin{definition}
Let $X$ be a standard Borel space and $G=(X,R)$ a Borel graph. A \emph{Borel partial matching} is a Borel subset of edges $P\subseteq R$ such
that every vertex of the graph $(X,P)$ has degree at most $1$. We write
$$
\dom(P)\ :=\ \{x\in X:\exists y\in X,\,\{x,y\}\in P\}
$$
 for the set of vertices matched by $P$. Moreover, we say that $P$ is a \emph{Borel perfect matching} if $\dom(P)=X$, or equivalently, if every vertex of the graph $(X, P)$ has degree exactly $1$.
\end{definition}

All graphs considered in this paper are generated by finitely many Borel
functions. We now fix some notation related to such graphs.

\begin{definition}
Let $X$ be a standard Borel space, and let
$f_1,\ldots,f_n:X\to X$
be Borel functions. We denote by
$G_{f_1,\ldots,f_n}=(X,R)$
the undirected graph generated by $f_1,\ldots,f_n$, where
$$
xRy
 \Longleftrightarrow 
x\neq y
\ \text{ and }\
\exists\,1\le i\le n\ \bigl[\, f_i(x)=y \mbox{ or } f_i(y)=x\,\bigr].
$$
The connectedness relation of $G_{f_1,\dots, f_n}$ is denoted by
$E_{f_1,\ldots,f_n}$, i.e., $(x,y)\in E_{f_1,\dots, f_n}$ if $x$ and $y$ belong to the same connected component of $G_{f_1,\dots, f_n}$.

\end{definition}

Similarly to the free part of a group action, we can also define the free part of the monoid action given by functions.

\begin{definition}
Let $X$ be a standard Borel space, and let
$f_1,\ldots,f_n:X\to X$
be commuting Borel functions. The \emph{free part} of the semigroup action
generated by $f_1,\ldots,f_n$ is
$$
\begin{aligned}
F(X):=
\bigl\{x\in X:\ &
\forall \alpha_1,\ldots, \alpha_n, \beta_1,\ldots, \beta_n\in\mathbb N,\\
&
(f_1^{\alpha_1}\circ\cdots\circ f_n^{\alpha_n})(x)
=
(f_1^{\beta_1}\circ\cdots\circ f_n^{\beta_n})(x)
\Longrightarrow
\alpha_i=\beta_i
\text{ for every }1\le i\le n
\bigr\}.
\end{aligned}
$$

\end{definition}
For notational simplicity, whenever there is no danger of confusion, we will write $f_1^{\alpha_1}\cdots f_n^{\alpha_n}(x)$ for
$(f_1^{\alpha_1}\circ\cdots\circ f_n^{\alpha_n})(x)$.

We next introduce two notions of forward marker sets. They will be used to construct
Borel sets which are sparse in the forward directions, but which 
can still be reached within a uniformly bounded number of forward steps from any point in the space. For the case of one function, these notions were introduced and studied in \cite{GrebikHiggins2026}.

\begin{definition}
Let $X$ be a standard Borel space, let
$f_1,\ldots,f_n:X\to X$
be commuting Borel functions, and let $M\in\mathbb N$, $r\in\mathbb N^+$. 
\begin{enumerate}
\item A Borel set
$H\subseteq F(X)$ is \emph{hitting} for
$G_{f_1,\ldots,f_n}\upharpoonright F(X)$ if, for every $x\in F(X)$, there exist
$\alpha_1,\ldots, \alpha_n\in\mathbb N$ such that
$f_1^{\alpha_1}\cdots f_n^{\alpha_n}(x)\in H$. Moreover, we say that $H$ has \emph{syndeticity} $M$ if for every $x\in F(X)$ there exist
$\alpha_1,\ldots, \alpha_n\in\mathbb N$ with $\sum_{i=1}^n \alpha_i\leq M$ such that
$f_1^{\alpha_1}\cdots f_n^{\alpha_n}(x)\in H$.
\item
 A Borel set
$H\subseteq F(X)$ is \emph{$r$-forward-independent} for
$G_{f_1,\ldots,f_n}\upharpoonright F(X)$ if, for all distinct $x,y\in H$, we have
$f_1^{\alpha_1}\cdots f_n^{\alpha_n}(x)
\neq
f_1^{\beta_1}\cdots f_n^{\beta_n}(y)$
whenever
$\alpha_i, \beta_i \in [0,r]\cap \mathbb{N}$
and
$\alpha_i\beta_i=0$
for every $1\le i\le n$. 
\end{enumerate}
\end{definition}

We will use the following result about the existence of Borel quai-kernels from \cite{Wang2026}.

\begin{lemma}[{\cite[Theorem 1.2]{Wang2026}}]
\label{lem:qk}
Let $X$ be a standard Borel space and let $\vec{G}=(X,\vec{R})$ be a locally countable Borel directed graph with finite Borel chromatic number
$\chi_B(\vec{G})<\infty$.
Then there exists a Borel set $X'\subseteq X$ which is a \emph{quasi-kernel} of $\vec{G}$, that is:
\begin{enumerate}[label=(\roman*),leftmargin=2.7em]
\item \textup{(Independent)} For all distinct $x,y\in X'$, neither $x\,\vec{R}\,y$ nor $y\,\vec{R}\,x$ holds.
\item \textup{(Hitting with syndeticity $2$)} For every $y\in X\setminus X'$ there exist $x\in X'$ and $z\in X$ such that
$$
y\,\vec{R}\,x
\mbox{ or }
(y\,\vec{R}\,z \mbox{ and } z\,\vec{R}\,x).
$$
\end{enumerate}
\end{lemma}

\subsection{Borel asymptotic dimension}

We next recall the notion of Borel asymptotic dimension and several consequences
that will be used later. Our formulation follows the framework of
\cite{ConleyJacksonMarksSewardTuckerDrob2023}.

Let $X$ be a standard Borel space. A function $\rho\colon X\times X\to [0, +\infty]$ is called an \emph{extended metric} if $\rho$ can take the value $+\infty$ but otherwise satisfies the defining properties of a metric. If $(X, \rho)$ is a standard Borel space equipped with an extended metric, then the \emph{finite-distance equivalence relation} is defined as $E_\rho=\{(x,y)\in X\times X\colon \rho(x,y)<+\infty\}$. For any graph $G=(X, R)$, we define the \emph{graph metric} $\rho_G$ on $X$ by letting $\rho_G(x,y)$ be the length of the shortest path in $G$ from $x$ to $y$. The graph metric $\rho_G$ is an extended metric and the finite-distance equivalence relation $E_{\rho_G}$ is exactly the connectedness relation of $G$. If $G$ is a Borel graph, then $\rho_G$ is a Borel extended metric on $X$.

A Borel extended metric $\rho$ on $X$
is called \emph{proper} if, for every $x\in X$ and every $r<\infty$, the ball
$B_\rho(x,r):=\{y\in X:\rho(x,y)\le r\}$
is finite. In particular, if $G$ is a locally finite Borel graph, then its
graph metric is a proper Borel extended metric.

In the following we recall the notion of finite Borel asymptotic dimension from \cite{ConleyJacksonMarksSewardTuckerDrob2023} and define a strengthening of it which we will use in the rest of this paper. The definition of the strengthened notion and the lemma following the definition were suggested by Shinko and Weilacher in  private communications.

\begin{definition}\label{def:fBasd}
Let $m\in\mathbb{N}$ and let $(X,\rho)$ be a standard Borel space equipped with a Borel extended
metric. 
\begin{enumerate}
\item[(1)] We say that $(X,\rho)$ has \emph{Borel asymptotic dimension at most
$m$} if, for every $r>0$, there exists a Borel equivalence relation $E_r$
on $X$ such that:
\begin{enumerate}[label=(\roman*),leftmargin=2.7em]
\item every $E_r$-class has uniformly bounded $\rho$-diameter, i.e., there is $d>0$ such that for every $E_r$-class $C$ and $x, y\in C$, $\rho(x,y)\leq d$;

\item for every $x\in X$, the ball
$B_\rho(x,r)$
meets at most $m+1$ many $E_r$-classes.
\end{enumerate}
We say that $(X,\rho)$ has \emph{finite Borel asymptotic dimension} if it has
Borel asymptotic dimension at most $m$ for some $m<\infty$. We write this as
$\operatorname{asdim}_B(X,\rho)<\infty$. 
\item[(2)] We say that $(X,\rho)$ has \emph{strong Borel asymptotic dimension at most
$m$} if, in part (1) of this definition, the equivalence relation $E_r$ is required to be smooth. 
We say that $(X,\rho)$ has \emph{finite strong Borel asymptotic dimension} if it has
strong Borel asymptotic dimension at most $m$ for some $m<\infty$. We write this as
$\operatorname{asdim}^*_B(X,\rho)<\infty$.
\end{enumerate}
\end{definition}

It is clear from the definition that if $\operatorname{asdim}^*_B(X, \rho)<\infty$, then $\operatorname{asdim}_B(X, \rho)<\infty$. Since finite Borel equivalence relations are smooth, the notion of Borel asymptotic dimension and that of strong Borel asymptotic dimension are the same if $(X, \rho)$ is proper.

The following result is a consequence of results from \cite{ConleyJacksonMarksSewardTuckerDrob2023}. It establishes finite strong Borel asymptotic dimension as a useful tool for proving hyperfiniteness.  

\begin{lemma} \label{lem:hyp}
Let $G=(X,R)$ be a locally countable Borel graph, and let $\rho_G$ be its graph
metric. If $\operatorname{asdim}^*_B(X,\rho_G)<\infty$,
then the connectedness relation of $G$ is hyperfinite.
\end{lemma}

\begin{proof} 
The proof is essentially the same as that of \cite[Theorem 7.1]{ConleyJacksonMarksSewardTuckerDrob2023}. Let $(X, \rho_G)$ have strong Borel aysmptotic dimension at most $m$. We inductively define $r_n\geq n$ and an increasing sequence of smooth Borel equivalence relations $F_n$ so that every $F_n$-class has uniformly bounded $\rho$-diameter and for every $x\in X$, the ball $B_\rho(x, r_n)$ meets at most $m+1$ many $F_n$-classes. To begin, we set $r_0=0$ and let $F_0$ be the equality relation on $X$. Assume inductively that for $n\geq 1$, $r_{n-1}$ and $F_{n-1}$ have been defined. Since $F_{n-1}$ is smooth, we have a Borel selector $s\colon X\to X$ for $F_{n-1}$. Choose $r_n\geq n$ so that for all $x\in X$, the $F_{n-1}$-class containing $x$ has $\rho$-diameter at most $r_n$. By the strong Borel asymptotic dimension assumption, there is a smooth Borel equivalence relation $E$ such that every $E$-class has uniformly bounded $\rho$-diameter and for every $x\in X$, $B_\rho(x, 2r_n)$ meets at most $m+1$ many $E$-classes. Let $t\colon X\to X$ be a Borel selector for $E$. Define $F_n=\{(x,y)\in X\times X\colon s(x)Es(y)\}$. Then $t\circ s$ is a Borel selector for $F_n$, and thus $F_n$ is smooth. It is easy to see that $F_{n-1}\subseteq F_n$ and that every $F_n$-class has uniformly bounded $\rho$-diameter. Moreover, for any $x\in X$, the number of $F_n$-classes meeting $B(x, r_n)$ is bounded by the number of $E$-classes meeting $B(x, 2r_n)\supseteq s(B(x, r_n))$, which is at most $m+1$. 

Now let $F_\infty=\bigcup_n F_n$. Then $F_\infty$ is hypersmooth. Since $F_\infty$ is a countable Borel equivalence relation, it follows that $F_\infty$ is hyperfinite \cite[Theorem 5.1 (3)]{DoughertyJacksonKechris1994}. Finally, $F_\infty$ is a subequivalence relation of the connectedness relation of $G$ with index $m+1$. By \cite[Proposition 1.3 (vii)]{JacksonKechrisLouveau2002}, the connectedness relation of $G$ is also hyperfinite.
\end{proof}

Besides hyperfiniteness, finite Borel asymptotic dimension also gives useful
bounds for Borel chromatic numbers. We shall use the following consequence.

\begin{lemma}[{\cite[Corollary 8.2]{ConleyJacksonMarksSewardTuckerDrob2023}}] \label{lem:color}
Let $X$ be a standard Borel space, and let $G$ be a locally finite Borel
graph on $X$ with graph metric $\rho_G$. If
$\operatorname{asdim}_B(X,\rho_G)<\infty$,
then
$\chi_B(G)\le 2\chi(G)-1$.
\end{lemma}

\section{Forward-Independent Hitting Sets}
In this section we prove the following theorem about the existence of forward-independent hitting sets for bounded-to-one commuting Borel functions. The theorem and an outline of its proof were suggested by Weilacher.

\begin{theorem} \label{thm:Weilacher} Let $X$ be a standard Borel space, and let $f_1, \dots, f_n\colon X\to X$ be commuting Borel functions each of which is bounded-to-one. Then for any $r>0$, there exists a Borel $r$-forward-independent hitting set for $G_{f_1,\dots, f_n}\upharpoonright F(X)$ with syndeticity $2n^2r$.
\end{theorem}

The rest of this section is devoted to a proof of Theorem~\ref{thm:Weilacher}. Our proof utilizes the deep connections between the theory of the LOCAL model of distributed computing (\cite{Lin92}) and descriptive combinatorics, which was first explored by Bernshteyn \cite{Ber23} and developed further for monoid actions in \cite{NaryshkinShinkoWeilacherYuCommutingFunctions}. Combining the results of \cite{Ber23} and \cite{NaryshkinShinkoWeilacherYuCommutingFunctions}, we are able to reduce our proof to that of a concrete statement in continuous combinatorics for $\mathbb{Z}^d$-actions.

We recall the following definitions from \cite{NaryshkinShinkoWeilacherYuCommutingFunctions}. Let $\mathsf{M}$ be a monoid. An \emph{$\mathsf{M}$-local algorithm} is a function $A\colon \Sigma^F\to \Lambda$, where $\Sigma, \Lambda$ are finite sets and $F\subseteq \mathsf{M}$ is a finite generating set containing the identity of $\mathsf{M}$. A \emph{locally checkable labeling problem} (or an \emph{LCL}) on $\mathsf{M}$ is an $\mathsf{M}$-local algorithm with $\Lambda=\{0,1\}$. If $\Pi$ is an LCL on $\mathsf{M}$, we also identify $\Pi$ with the set $\Pi^{-1}(1)$. Now consider an action of $\mathsf{M}$ on a space $Y$. If $\Pi\colon \Sigma^F\to \Lambda$ is an LCL on $\mathsf{M}$ and $y\in Y$, a function $\ell\colon Y\to \Sigma$ is a \emph{$\Pi$-labeling at} $y$ if $\Pi\big(f\mapsto \ell(fy)\big)=1$ or $\big(f\mapsto \ell(fy)\big)\in \Pi$ under the identification of $\Pi$ as a set. We say that $\ell$ is a \emph{$\Pi$-labeling} of $Y$ if $\ell$ is a $\Pi$-labeling at every $y\in Y$. 

In our context, let $\mathsf{M}$ be the monoid $\mathbb{N}^n$. Then the functions $f_1, \dots, f_n\colon X\to X$ give an action of $\mathsf{M}$ on $X$. Fix any interger $r>0$. Let $F=\big([0,2n^2r]\cap \mathbb{N}\big)^n$ and $\Sigma=\{0,1\}$. Let $\Pi$ be the set of all $\sigma\colon F\to \Sigma$ such that the following conditions hold:
\begin{enumerate}
\item[(i)] there is $(\gamma_1, \dots, \gamma_n)\in F$ such that $\sum_{i=1}^n\gamma_i\leq 2n^2r$ and $\sigma(\gamma_1,\dots, \gamma_n)=1$; and
\item[(ii)] for all distinct $(\gamma_1, \dots, \gamma_n), (\lambda_1, \dots, \lambda_n)\in F$, if 
$$ \sigma(\gamma_1,\dots, \gamma_n)=\sigma(\lambda_1,\dots,\lambda_n)=1, $$
then we have
$$ (\gamma_1+\alpha_1, \dots, \gamma_n+\alpha_n)\neq (\lambda_1+\beta_1,\dots, \lambda_n+\beta_n) $$
whenever $\alpha_i, \beta_i\in [0,r]\cap \mathbb{N}$ and $\alpha_i\beta_i=0$ for every $1\leq i\leq n$.
\end{enumerate}
It is then straightforward to check that a set $H\subseteq F(X)$ is an $r$-forward-independent hitting set with syndeticity $2n^2r$ if and only if $H$ is a $\Pi$-labeling of $F(X)$.

We will be using the following theorems from \cite{NaryshkinShinkoWeilacherYuCommutingFunctions}.

\begin{theorem}[{\cite[Theorem 1.7]{NaryshkinShinkoWeilacherYuCommutingFunctions}}]\label{thm:3.2} Let $\Gamma$ be an abelian group and $\mathsf{M}\subseteq \Gamma$ be a finitely generated submonoid with $\Gamma=\langle \mathsf{M}\rangle$.  Let $\Pi$ be an LCL on $\mathsf{M}$. Then the following are equivalent:
\begin{enumerate}
\item Every free continuous action of $\mathsf{M}$ on a $0$-dimensional Polish space admits a continuous $\Pi$-labeling.
\item Every free continuous clopen-preserving bounded-to-one action of $\mathsf{M}$ on a $0$-dimensional Polish space admits a continuous $\Pi$-labeling.
\item $\Pi$ can be solved by an $O(\log^*(n))$-local algorithm on $\Gamma$.
\item $\Pi$ can be solved by an $O(\log^*(n))$-local algorithm on $\mathsf{M}$.
\end{enumerate}
\end{theorem}

\begin{theorem}[{\cite[Proposition 3.22]{NaryshkinShinkoWeilacherYuCommutingFunctions}}]\label{thm:3.3} Let $\mathsf{M}$ be a finitely generated monoid and let $\Pi$ be an LCL on $\mathsf{M}$. If $\Pi$ can be solved by an $O(\log^*(n))$-local algorithm on $\mathsf{M}$, then any free bounded-to-one Borel action of $\mathsf{M}$ on a standard Borel space admits a Borel $\Pi$-labeling.
\end{theorem}

By Theorem~\ref{thm:3.3} and the equivalence of Theorem~\ref{thm:3.2} (1) and (4), it suffices for us to show the following result.

\begin{proposition}\label{prop:3.4} Let $\mathbb{Z}^n$ act freely and continuously on a $0$-dimensional Polish space $Z$. Let $f_1, \dots, f_n\colon Z\to Z$ be defined by $f_i(z)=e_i\cdot z$ for $1\leq i\leq n$ and $z\in Z$, where $e_i=(0,\dots, i, \dots, 0)$ is the $i$-th generator of $\mathbb{Z}^n$. Then for any $r>0$, there is a clopen set $H\subseteq Z$ which is an $r$-forward-independent hitting set with syndeticity $2n^2r$.
\end{proposition}

Before giving a proof of Proposition~\ref{prop:3.4}, 
we briefly discuss the notion of maximal $r$-independent sets in a locally finite Borel graph of bounded degree. Let $G=(X, R)$ be a locally finite Borel graph with bounded degree.
Let $\rho_G$ be the graph metric on $G$. For $r>0$, a set $P\subseteq X$ is \emph{$r$-independent} if for any $x, y\in P$,
$\rho_G(x,y)>r$. An $r$-independent set $P\subseteq X$ is \emph{maximal} if for any $x\in X$ there is $y\in P$ such that $\rho_G(x,y)\leq r$. In the context of functions $f_1, \dots, f_n\colon X\to X$ and $r>0$, there is again an LCL $\Pi$ on the monoid $\mathbb{N}^n$ such that any maximal $r$-independent set in $G_{f_1,\dots, f_n}\upharpoonright F(X)$ is a $\Pi$-labeling of $F(X)$.

\begin{lemma}\label{lem:3.5} Let $\mathbb{Z}^d$ act freely and continuously on a $0$-dimensional Polish space $Z$. Let $G$ be the Schreier graph of the action on $Z$. Then for any $r>0$, there is a clopen set $H\subseteq Z$ which is a maximal $r$-independent set of $G$.
\end{lemma}

\begin{proof} There are many proofs of this fact. One way to see it is to use Theorem~\ref{thm:3.2} and work of Linial \cite{Lin92}, who showed that there is an $O(\log^*(n))$-local algorithm for finding a maximal $1$-independent set in graphs of bounded degree (also see \cite[Proposition 6.2]{NaryshkinShinkoWeilacherYuCommutingFunctions}). 
To apply Linial's result, we consider a free continous open bounded-to-one action of $\mathbb{N}^d$ on a $0$-dimensional Polish space $Z$ and consider the induced Schreier graph $G$. Furthermore, consider the graph $G^r=\{(x,y)\in Z\colon \rho_G(x,y)\leq r\}$, where $\rho_G$ is the graph metric on $G$. Then a subset $P\subseteq Z$ is $r$-independent in $G$ if and only if it is $1$-independent in $G^r$. Since $G$ has bounded degree, so does $G^r$. Hence Linial's result applies to $G^r$,   yielding (3) of Theorem~\ref{thm:3.2}.
\end{proof}

\begin{proof}[Proof of Proposition~\ref{prop:3.4}] Fix any $r>0$. Let $H\subseteq Z$ be a clopen set which is also a maximal $nr$-independent set of $G=G_{f_1,\dots, f_n}$ given by Lemma~\ref{lem:3.5}. Let $\rho$ be the graph metric on $G$. We verify that $H$ is an $r$-forward-independent hitting set with syndeticity $2n^2r$. 

Take an arbitrary $z\in Z$. Let $u=f_1^{nr}\cdots f_n^{nr}(z)=(nr, nr,\dots, nr)\cdot z$. By the maximality of $H$, there is $y\in H$ such that $\rho(u,y)\leq nr$. Thus there exists $(\alpha_1, \dots, \alpha_n)\in \big([-nr, nr]\cap \mathbb{Z}\big)^n$ such that $(\alpha_1,\dots, \alpha_n)\cdot u=y$. It follows that $(nr+\alpha_1,\dots, nr+\alpha_n)\cdot z=y\in H$. Since $(nr+\alpha_1,\dots, nr+\alpha_n)\in \big([0,2nr]\cap \mathbb{N}\big)^n$ and $\sum_{i=1}^n (nr+\alpha_i)\leq 2n^2r$, we have that $H$ is a hitting set with syndeticity $2n^2r$.

It remains to check that $H$ is $r$-forward-independent. For this, assume $x, y$ are distinct and for some $\alpha_1,\dots, \alpha_n, \beta_1,\dots, \beta_n\in [0,r]\cap\mathbb{N}$ with $\alpha_i\beta_i=0$ for every $1\leq i\leq n$, we have $(\alpha_1,\dots, \alpha_n)\cdot x=(\beta_1,\dots, \beta_n)\cdot y$. Then $\rho(x,y)\leq nr$. Thus $x, y$ cannot be in $H$ simultaneously. 
\end{proof}

\section{Structural Theorems}

In this section we develop a geometric decomposition theory underlying the rest of the paper.
We first introduce some definitions adapted to the commuting family $\{f_1,\cdots,f_n\}$ for $n\in \mathbb{N}$.
We then construct a refined family of regions whose relevant faces are uniformly separated, and use this family to obtain marker decompositions into rootless and rooted regions.

Throughout this section, let $X$ be a standard Borel space, and let
$f_1,\ldots,f_n:X\to X$
be commuting Borel functions. We work on the free part
$F(X)\subseteq X$
of the graph $G_{f_1,\ldots,f_n}$.

We begin with a family of pseudo-distances which measure displacement in each  $f_i$-direction. Fix $1\leq i\leq n$.

\begin{definition}[$f_i$-pseudo-distance]\label{def:fi-pseudo-dist}
For $x,y\in F(X)$, if there exist $\alpha_1,\cdots , \alpha_n, \beta_1,\cdots , \beta_n\in\mathbb N$ such that
$$
f_1^{\alpha_1} \cdots  f_n^{\alpha_n}(x)=f_1^{\beta_1} \cdots f_n^{\beta_n}(y),
$$
then we define the \emph{$f_i$-pseudo-distance} between $x$ and $y$ by
$$
\rho_{f_i}(x,y):=|\alpha_i-\beta_i|.
$$
Otherwise, we set
$\rho_{f_i}(x,y):=\infty.$

\end{definition}

Since we work on $F(X)$, this value is unique, and thus $\rho_{f_i}$ is well defined.

For $A,B\subseteq F(X)$, we define the \emph{$f_i$-pseudo-distance} between $A$ and $B$ by
$$
\rho_{f_i}(A,B):=\inf\{\rho_{f_i}(x,y):x\in A,\ y\in B\}.
$$
And $\rho_{f_i}(x,A)=\rho_{f_i}(\{x\},A)$.
These pseudo-distances will be used to describe relative positions of points and regions. 
We say that $x$ is \emph{$f_i$-before} $y$ if $\rho_{f_i}(x,y)$ is finite and 
$$
\rho_{f_i}(x,y)>\rho_{f_i}(x,f_i(y)).
$$ 

For $A,B\subseteq F(X)$, we say that $A$ is \emph{$f_i$-before} $B$ if for all $x\in A$ and $y\in B$, the point $x$ is $f_i$-before $y$.
 
 An \emph{$f_i$-edge} is an edge of the form $\{x,f_i(x)\}$ for some $x\in F(X)$.
 
\begin{definition}[$f_i$-diameter]
Let $A\subseteq F(X)$. Define
$$
\diam_{f_i}(A):=\sup\bigl\{\alpha\in\mathbb N:\ \exists x\in A \text{ such that } f_i^{\alpha}(x)\in A\bigr\}.
$$
We call $\diam_{f_i}(A)$ the \emph{$f_i$-diameter} of $A$. 
\end{definition}

The marker regions constructed below have finite $f_i$-diameter
for every $1\le i\le n$. Next we specify the geometric shapes that will appear as marker regions.

\begin{definition}[Rooted and rootless regions]
A \emph{marker region} is a nonempty subset $A\subseteq F(X)$ such that
there exist a unique point
$x=\operatorname{root}(A)\in F(X)$,
called the \emph{root} of $A$, and integers
$\beta_1,\ldots, \beta_n\in\mathbb{N}$
satisfying
$$
\begin{aligned}
A=\Bigl\{y\in F(X):\ &
\exists\, \alpha_1,\ldots, \alpha_n\in\mathbb N
\text{ with }0\le \alpha_i\le \diam_{f_i}(A)\text{ for every $1\le i\le n$, and}\\
&
\quad \quad \quad \quad f_1^{\beta_1+\alpha_1}\cdots f_n^{\beta_n+\alpha_n}(y)=x
\Bigr\}.
\end{aligned}
$$
The region $A$ is called a \emph{rooted region} if $\beta_1=\cdots=\beta_n=0$; it is called a \emph{rootless region} if $\beta_i>0$ for all $1\leq i\leq n$.
\end{definition}

A marker region does not necessarily contain its root. The integer $\beta_i$ in the definition records the distance from the root to
the region in the $f_i$-direction. More precisely,
$$
\beta_i=\rho_{f_i}(\{\operatorname{root}(A)\},A)
\qquad\text{for every }1\le i\le n.
$$

It is natural to speak of faces of a marker region. The following is the precise definition. It can be applied to regions of more general shapes.

\begin{definition}[Faces of a region]\label{def:faces}
For $A\subseteq F(X)$ and each $1\le i\le n$, we define the \emph{forward} and \emph{backward} $f_i$-faces of $A$ by
$$
F_i^+(A):=\{x\in A:\ f_i(x)\notin A\},
\qquad
F_i^-(A):=\{x\in A:\ x\notin f_i(A)\}.
$$
We also write
$$
F_i(A):=F_i^+(A)\cup F_i^-(A).
$$
The family
$$
\{F_i^+(A),\,F_i^-(A):1\le i\le n\}
$$
will be referred to as the $2n$ \emph{faces} of $A$.
\end{definition}

Figure~\ref{Fig.2} illustrates the four faces of a marker region $A$ for $n=2$, where $f_1$ is injective and $f_2$ is
countable-to-one.
\begin{figure}[H]
    \centering
    \tikzset{every picture/.style={line width=0.75pt}}

\begin{tikzpicture}[x=0.75pt,y=0.75pt,yscale=-1,xscale=1]
% uncomment if required: \path (0,374);

% =========================================================
% 原第二幅图 -> 现在放到右边（整体右移 198）
% =========================================================
\begin{scope}[shift={(198,0)}]

% 第一幅图前面那块绿色填充曲线（保留）
\draw [
  color={rgb,255:red,0; green,0; blue,0},
  draw opacity=1,
  fill={rgb,255:red,184; green,233; blue,134},
  fill opacity=1
]
(274.71,30.59) .. controls (245.71,42.64) and (275.86,61.93) .. (274.71,61.31);

% 左侧绿色区域
\path[
  fill={rgb,255:red,184; green,233; blue,134},
  fill opacity=1,
  draw=none
]
(262,43.15)
-- (274.71,61.31)
.. controls (281,65.17) and (288,67.5) .. (297.21,69.71)
-- (297.08,190.30)
.. controls (272.67,185.17) and (263.33,171.50) .. (262,163.75)
-- cycle;

% 右上角黄色小区域
\path[
  fill={rgb,255:red,255; green,228; blue,110},
  fill opacity=1,
  draw=none
]
(394.76,29.82)
.. controls (389.58,33.80) and (392.52,35.49) .. (398.53,36.13)
-- (394.76,37.50)
-- cycle;

% 右侧黄色长条区域
\path[
  fill={rgb,255:red,255; green,228; blue,110},
  fill opacity=1,
  draw=none
]
(398.53,36.13)
-- (398.53,155.05)
.. controls (392.52,154.42) and (389.58,154.13) .. (392.10,153.10)
-- (392.10,33.10)
-- cycle;

% 主体描边
\draw (274.71,30.59) -- (394.76,29.82);
\draw (297.21,69.71) -- (398.53,36.13);
\draw (297.08,190.30) -- (398.53,155.05);

\draw [color={rgb,255:red,0; green,0; blue,0}, draw opacity=1]
  (274.71,30.59) .. controls (245.71,42.64) and (275.86,61.93) .. (274.71,61.31);

\draw (274.71,61.31) .. controls (281,65.17) and (288,67.5) .. (297.21,69.71);
\draw (262,43.15) -- (262,163.75);
\draw (297.21,69.71) -- (297.08,190.30);
\draw (262,163.75) .. controls (263.33,171.50) and (272.67,185.17) .. (297.08,190.30);

\draw [dash pattern={on 2.25pt off 1.5pt}] (274.71,61.31) -- (274.71,150.15);
\draw [dash pattern={on 1.5pt off 0.75pt}]
  (274.71,150.15) .. controls (268.67,152.83) and (263.67,153.83) .. (262,163.75);

\draw (394.76,29.82) .. controls (389.58,33.80) and (392.52,35.49) .. (398.53,36.13);
\draw (398.53,36.13) -- (398.53,155.05);
\draw [dash pattern={on 1.5pt off 0.75pt}]
  (394.76,150.15) .. controls (389.58,154.13) and (392.52,154.42) .. (398.53,155.05);

% 双线透视效果保留
\draw [dash pattern={on 2.25pt off 1.5pt}] (394.76,150.15) -- (394.76,37.5);
\draw (394.76,29.82) -- (394.76,35.07);
\draw [dash pattern={on 2.25pt off 1.5pt}] (392.1,153.1) -- (392.1,33.1);

\draw [dash pattern={on 2.25pt off 1.5pt}] (274.71,150.15) -- (394.76,150.15);
\draw [dash pattern={on 2.25pt off 1.5pt}] (274.71,30.59) -- (274.71,61.31);

% 这幅图自己的标签和小曲线
\draw (401,90.4) node [anchor=north west][inner sep=0.75pt]
  [font=\scriptsize,color={rgb,255:red,0; green,0; blue,0},opacity=1] {$F_{2}^{+}(A)$};

\draw (225,90.4) node [anchor=north west][inner sep=0.75pt]
  [font=\scriptsize,color={rgb,255:red,0; green,0; blue,0},opacity=1] {$F_{2}^{-}(A)$};

\draw (331,92.4) node [anchor=north west][inner sep=0.75pt]
  [font=\scriptsize] {$A$};

% 绿色
\draw [color={rgb,255:red,124; green,211; blue,33}, draw opacity=1]
  (240,90) .. controls (252.6,78.2) and (269,81.8) .. (270,90);

% 黄色
\draw [color={rgb,255:red,251; green,201; blue,28}, draw opacity=1]
  (400,90) .. controls (390,81.86) and (415,83.57) .. (420,90);

\end{scope}

% =========================================================
% 原第三幅图 -> 现在放到左边（整体左移 198）
% =========================================================
\begin{scope}[shift={(-198,0)}]

% 上顶面
\path[
  fill={rgb,255:red,244; green,159; blue,146},
  fill opacity=1,
  draw=none
]
(472.71,30.77) -- (592.76,30.00)
.. controls (587.58,33.98) and (590.52,35.68) .. (596.53,36.31)
-- (495.21,69.89)
.. controls (486.00,67.68) and (479.00,65.35) .. (472.71,61.49)
.. controls (473.86,62.11) and (443.71,42.82) .. (472.71,30.77)
-- cycle;

% 下底面
\path[
  fill={rgb,255:red,207; green,231; blue,246},
  fill opacity=1,
  draw=none
]
(472.71,150.33) -- (592.76,150.33)
.. controls (587.58,154.31) and (590.52,154.60) .. (596.53,155.23)
-- (495.08,190.48)
.. controls (470.67,185.35) and (461.33,171.68) .. (460.00,163.93)
.. controls (461.67,154.01) and (466.67,153.01) .. (472.71,150.33)
-- cycle;

% 主体轮廓描边
\draw [color={rgb,255:red,0; green,0; blue,0}, draw opacity=1]
  (472.71,30.77) .. controls (443.71,42.82) and (473.86,62.11) .. (472.71,61.49);
\draw (472.71,30.77) -- (592.76,30);
\draw (495.21,69.89) -- (596.53,36.31);

\draw (460,163.93) .. controls (461.33,171.68) and (470.67,185.35) .. (495.08,190.48);
\draw (495.08,190.48) -- (596.53,155.23);

\draw (592.76,30) .. controls (587.58,33.98) and (590.52,35.68) .. (596.53,36.31);
\draw (596.53,36.31) -- (596.53,155.23);
\draw [dash pattern={on 1.5pt off 0.75pt}]
  (592.76,150.33) .. controls (587.58,154.31) and (590.52,154.6) .. (596.53,155.23);

\draw [dash pattern={on 2.25pt off 1.5pt}] (592.76,150.33) -- (592.76,35.25);
\draw (592.76,30) -- (592.76,35.25);

\draw (472.71,61.49) .. controls (479,65.35) and (486,67.68) .. (495.21,69.89);
\draw [dash pattern={on 2.25pt off 1.5pt}] (472.71,61.49) -- (472.71,150.33);
\draw [dash pattern={on 1.5pt off 0.75pt}]
  (472.71,150.33) .. controls (466.67,153.01) and (461.67,154.01) .. (460,163.93);

\draw (495.21,69.89) -- (495.08,190.48);
\draw [dash pattern={on 2.25pt off 1.5pt}] (472.71,150.33) -- (592.76,150.33);
\draw [dash pattern={on 2.25pt off 1.5pt}] (472.71,30.77) -- (472.71,61.49);

% 改成严格竖直
\draw (460,43.34) -- (460,163.93);

% 中间竖虚线
\draw [dash pattern={on 2.25pt off 1.5pt}] (590,153.1) -- (590,33.1);

% 这幅图自己的标签和小曲线
\draw (551,172.4) node [anchor=north west][inner sep=0.75pt]
  [font=\scriptsize,color={rgb,255:red,0; green,0; blue,0},opacity=1] {$F_{1}^{-}(A)$};

\draw (521,10.4) node [anchor=north west][inner sep=0.75pt]
  [font=\scriptsize,color={rgb,255:red,0; green,0; blue,0},opacity=1] {$F_{1}^{+}(A)$};

\draw (531,92.4) node [anchor=north west][inner sep=0.75pt]
  [font=\scriptsize] {$A$};

% 红色
\draw [color={rgb,255:red,220; green,46; blue,31}, draw opacity=1]
  (510,40) .. controls (511.29,31.86) and (511.57,21.86) .. (520,20);

% 蓝色
\draw [color={rgb,255:red,119; green,181; blue,251}, draw opacity=1]
  (540,170) .. controls (537.29,179.57) and (544.71,178.71) .. (550,180);

\end{scope}

\end{tikzpicture} 
    \caption{The four faces of a marker region $A$.}
    \label{Fig.2}
\end{figure}
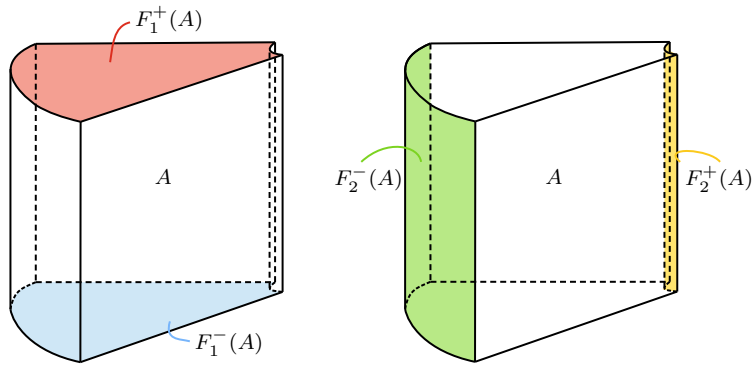

The reader should be cautioned that the marker regions illustrated in Figure~\ref{Fig.2} are idealized in that they represent better the case that all functions are surjective. In the case some functions fail to be surjective, some parts might be missing, and the regions would contain defects. However, we would still like to speak of the side lengths of a marker region as follows.  

\begin{definition}[$f_i$-side length]
Let $A\subseteq F(X)$ be a marker region,  and let $1\le i\le n$ and $d\in\mathbb N$. We define the notion of \emph{$f_i$-side length} of $A$, denoted $d_i(A)$, by defining two relations $d_i(A)\geq d$ and $d_i(A)\leq d$ as follows.
We write
$d_i(A)\ge d$
if for every $0\le \alpha\le d$,
$$ (f_i^\alpha)^{-1}\bigl(F_i^+(A)\bigr)\subseteq A;$$
we write
$d_i(A)\le d$
if
$\diam_{f_i}(A)\le d$,
and write
$d_i(A)=d$
if both $d_i(A)\ge d$ and $d_i(A)\le d$ hold.
\end{definition}

% More generally, for a nonempty set $S\subseteq\mathbb N$, we write
% $
% d_i(A)\in S
% $
% if $d_i(A)=d$ for some $d\in S$.

When $f_i$ is surjective, there is a unique $d\in\mathbb N$ such that
$d_i(A)=\diam_{f_i}(A)=d.
$
If $f_i$ is not surjective, the relation $d_i(A)=d$ may hold for more
than one value of $d$. This occurs when the region terminates in the
backward $f_i$-direction because further $f_i$-preimages do not exist.

Figure~\ref{Fig.3} illustrates two examples for $n=2$, where $f_1$ is injective, $f_2$ is
countable-to-one, and both are surjective. 
The left panel shows a rooted region $A$ with
$$
d_1(A)=2,\qquad d_2(A)=2.
$$
The right panel shows a
rootless region $B$ with
$$
d_1(B)=1,\qquad d_2(B)=1,
\qquad
\rho_{f_i}(\{\operatorname{root}(B)\},B)=1
\quad (1\le i\le 2).
$$
If any of these functions is not surjective, some illustrated points may not exist.

\begin{figure}[H]
    \centering
    \input{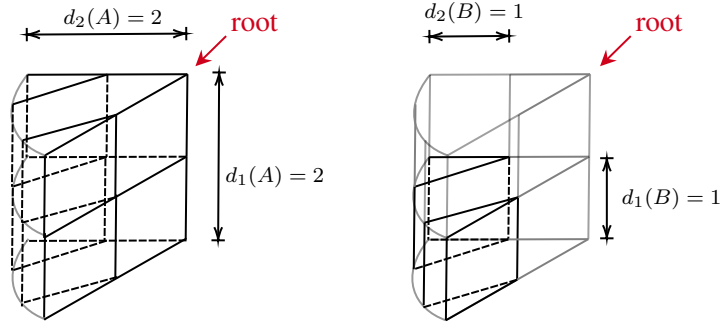}
    \caption{A rooted region and a rootless region.}
    \label{Fig.3}
\end{figure}

%Next we introduce the faces of a region. These are the boundary pieces through which the region exits in each forward or backward $f_i$-direction.

\subsection{Construction of face-separated regions}

We now begin the marker construction. The goal of this subsection is to start
from a coarse family of large regions and then refine it into a family whose
relevant faces are uniformly separated. The resulting face-separation property
is the key geometric input in the proof of the following theorem.

\begin{theorem}\label{thm:markers}
Let $X$ be a standard Borel space, and let $D\in\mathbb{N}^+$. For $n\in\mathbb{N}$, suppose that $f_1,\cdots ,f_n:X\to X$ are commuting Borel functions each of which is countable-to-one. Assume that for every $r\in\mathbb N^+$, there is a Borel $r$-forward-independent hitting set with syndeticity $Cr$ for $G_{f_1,\ldots,f_n}\upharpoonright F(X)$, where $C\in\mathbb N^+$ is independent of $r$. Then there exists a smooth Borel subequivalence relation
$R_D\subseteq E_{f_1,\cdots,f_n}\upharpoonright F(X)$
such that every $R_D$-class $A$ is a rootless region satisfying
$d_i(A)\ge D$ for $1\le i \le n$. Moreover, if $f_1^D\cdots f_n^D(x)=f_1^D\cdots f_n^D(y)$, then $ xR_D y$.
\end{theorem}

The theorem will be proved in the next subsection. In this subsection, we first construct a coarse cover of $F(X)$ by large
rootless regions. These initial regions are much larger than the final
$R_D$-classes, leaving enough space for the later refinements and boundary
adjustments. For the rest of this subsection, we work under the assumptions of
Theorem~\ref{thm:markers}.

Let $\rho$ be the graph metric on
$G_{f_1,\ldots,f_n}\upharpoonright F(X)$. Choose integers $D_1\gg D'=D+2$. We may further assume that $D_1$ is divisible by $4$. Let $H\subseteq F(X)$ be a Borel $D_1$-forward-independent hitting set with syndeticity $CD_1$ for
$G_{f_1,\ldots,f_n}\upharpoonright F(X)$, where $C\in\mathbb N^+$ is a constant.

\medskip
For each $x\in H$, define
$$
\begin{aligned}
T^{(0)}_{x}
:=\Bigl\{y\in F(X):\ &\exists \alpha_1,\cdots, \alpha_n\in(D_1,2CD_1]\cap\mathbb{N} \mbox{ such that } f_1^{\alpha_1}\cdots f_n^{\alpha_n}(y)=x\Bigr\}.
\end{aligned}
$$
It is easy to see that $T_x^{(0)}$ is a rootless region. Figure~\ref{Fig.4} illustrates $T_x^{(0)}$ with its parameters for $n=2$, where $f_1$ is injective and $f_2$ is
countable-to-one.
\begin{figure}[H]
    \centering
    \tikzset{every picture/.style={line width=0.75pt}} %set default line width to 0.75pt        

\begin{tikzpicture}[x=0.75pt,y=0.75pt,yscale=-1,xscale=1]
%uncomment if require: \path (0,374); %set diagram left start at 0, and has height of 374

%Straight Lines [id:da8640603933496599] 
\draw [line width=0.75]    (266.19,50.14) -- (386.23,49.38) ;
%Straight Lines [id:da5272087049904] 
\draw [line width=0.75]    (288.69,89.27) -- (390,55.68) ;
%Curve Lines [id:da35554828995569143] 
\draw [color={rgb, 255:red, 0; green, 0; blue, 0 }  ,draw opacity=1 ][line width=0.75]    (266.19,50.14) .. controls (237.2,60.16) and (263.7,86.42) .. (288.69,89.27) ;
%Straight Lines [id:da283812681849543] 
\draw [line width=0.75]  [dash pattern={on 2.25pt off 1.5pt}]  (266.19,169.7) -- (314.67,169.7) -- (337.07,169.7) -- (386.23,169.7) ;
%Straight Lines [id:da9219201814260687] 
\draw [line width=0.75]    (288.56,209.86) -- (390,174.61) ;
%Straight Lines [id:da02703391222385687] 
\draw [color={rgb, 255:red, 128; green, 128; blue, 128 }  ,draw opacity=1 ][line width=0.75]    (404.43,31.34) -- (405.3,169.53) ;
%Straight Lines [id:da26449166643386046] 
\draw [line width=0.75]    (288.69,89.27) -- (288.56,209.86) ;
%Straight Lines [id:da5713247775416893] 
\draw [line width=0.75]  [dash pattern={on 2.25pt off 1.5pt}]  (266.19,50.14) -- (266.19,169.7) ;
%Curve Lines [id:da7673470329675363] 
\draw [line width=0.75]    (386.23,49.38) .. controls (381.05,53.36) and (383.99,55.05) .. (390,55.68) ;
%Straight Lines [id:da007792629447112898] 
\draw [color={rgb, 255:red, 100; green, 99; blue, 99 }  ,draw opacity=1 ][line width=0.75]    (266.8,31.62) -- (404.43,31.34) ;
%Straight Lines [id:da23802924185274166] 
\draw [color={rgb, 255:red, 128; green, 128; blue, 128 }  ,draw opacity=1 ][line width=0.75]    (390,174.61) -- (405.3,169.53) ;
%Straight Lines [id:da047883369409288834] 
\draw [line width=0.75]    (390,55.68) -- (390,174.61) ;
%Curve Lines [id:da5036675100679523] 
\draw [line width=0.75]  [dash pattern={on 1.5pt off 0.75pt}]  (386.23,169.7) .. controls (384,171.42) and (383.28,172.45) .. (383.65,173.11) .. controls (384.14,174) and (386.57,174.25) .. (390,174.61) ;
%Straight Lines [id:da7651742319213847] 
\draw [color={rgb, 255:red, 128; green, 128; blue, 128 }  ,draw opacity=1 ][line width=0.75]    (289.3,70.75) -- (404.43,31.34) ;
%Straight Lines [id:da8213592548315297] 
\draw [color={rgb, 255:red, 128; green, 128; blue, 128 }  ,draw opacity=1 ][line width=0.75]    (386.23,169.7) -- (405.3,169.53) ;
%Curve Lines [id:da48562667428663053] 
\draw [line width=0.75]    (266.19,169.7) .. controls (237.2,179.72) and (263.57,207.01) .. (288.56,209.86) ;
%Straight Lines [id:da43089726020584207] 
\draw    (254,63) -- (254,183.6) ;
%Straight Lines [id:da8020915805267916] 
\draw  [dash pattern={on 2.25pt off 1.5pt}]  (386.23,169.7) -- (386.23,54.62) ;
%Straight Lines [id:da8509751735121364] 
\draw    (386.23,49.38) -- (386.23,51.7) -- (386.23,52.38) -- (386.23,52.6) -- (386.23,52.73) -- (386.23,53.1) -- (386.23,53.3) -- (386.23,53.46) -- (386.23,53.63) -- (386.23,54.63) ;
%Shape: Circle [id:dp5800921358364999] 
\draw  [color={rgb, 255:red, 208; green, 2; blue, 27 }  ,draw opacity=1 ][fill={rgb, 255:red, 208; green, 2; blue, 27 }  ,fill opacity=1 ][line width=1.5]  (406.1,31.34) .. controls (406.1,30.42) and (405.35,29.67) .. (404.43,29.67) .. controls (403.51,29.67) and (402.76,30.42) .. (402.76,31.34) .. controls (402.76,32.26) and (403.51,33) .. (404.43,33) .. controls (405.35,33) and (406.1,32.26) .. (406.1,31.34) -- cycle ;
%Straight Lines [id:da694714922830436] 
\draw    (316,20) -- (386,20) ;
\draw [shift={(386,20)}, rotate = 180] [color={rgb, 255:red, 0; green, 0; blue, 0 }  ][line width=0.75]    (0,3.35) -- (0,-3.35)(6.56,-1.97) .. controls (4.17,-0.84) and (1.99,-0.18) .. (0,0) .. controls (1.99,0.18) and (4.17,0.84) .. (6.56,1.97)   ;
%Straight Lines [id:da20762420652454283] 
\draw    (316,20) -- (264,20) ;
\draw [shift={(264,20)}, rotate = 360] [color={rgb, 255:red, 0; green, 0; blue, 0 }  ][line width=0.75]    (0,3.35) -- (0,-3.35)(6.56,-1.97) .. controls (4.17,-0.84) and (1.99,-0.18) .. (0,0) .. controls (1.99,0.18) and (4.17,0.84) .. (6.56,1.97)   ;
%Straight Lines [id:da6114292350695646] 
\draw    (244,110) -- (244,170) ;
\draw [shift={(244,170)}, rotate = 270] [color={rgb, 255:red, 0; green, 0; blue, 0 }  ][line width=0.75]    (0,3.35) -- (0,-3.35)(6.56,-1.97) .. controls (4.17,-0.84) and (1.99,-0.18) .. (0,0) .. controls (1.99,0.18) and (4.17,0.84) .. (6.56,1.97)   ;
%Straight Lines [id:da4539789011910994] 
\draw    (244,110) -- (244,30) ;
\draw [shift={(244,30)}, rotate = 90] [color={rgb, 255:red, 0; green, 0; blue, 0 }  ][line width=0.75]    (0,3.35) -- (0,-3.35)(6.56,-1.97) .. controls (4.17,-0.84) and (1.99,-0.18) .. (0,0) .. controls (1.99,0.18) and (4.17,0.84) .. (6.56,1.97)   ;
%Straight Lines [id:da4353925657467098] 
\draw    (412,41) -- (412,31) ;
\draw [shift={(412,31)}, rotate = 90] [color={rgb, 255:red, 0; green, 0; blue, 0 }  ][line width=0.75]    (0,3.35) -- (0,-3.35)(6.56,-1.97) .. controls (4.17,-0.84) and (1.99,-0.18) .. (0,0) .. controls (1.99,0.18) and (4.17,0.84) .. (6.56,1.97)   ;
%Straight Lines [id:da31138218589483857] 
\draw    (412,41) -- (412,50) ;
\draw [shift={(412,50)}, rotate = 270] [color={rgb, 255:red, 0; green, 0; blue, 0 }  ][line width=0.75]    (0,3.35) -- (0,-3.35)(6.56,-1.97) .. controls (4.17,-0.84) and (1.99,-0.18) .. (0,0) .. controls (1.99,0.18) and (4.17,0.84) .. (6.56,1.97)   ;
%Straight Lines [id:da7392327593188647] 
\draw    (390,20) -- (386,20) ;
\draw [shift={(386,20)}, rotate = 360] [color={rgb, 255:red, 0; green, 0; blue, 0 }  ][line width=0.75]    (0,3.35) -- (0,-3.35)(6.56,-1.97) .. controls (4.17,-0.84) and (1.99,-0.18) .. (0,0) .. controls (1.99,0.18) and (4.17,0.84) .. (6.56,1.97)   ;
%Straight Lines [id:da7713806085402943] 
\draw    (390,20) -- (405,20) ;
\draw [shift={(405,20)}, rotate = 180] [color={rgb, 255:red, 0; green, 0; blue, 0 }  ][line width=0.75]    (0,3.35) -- (0,-3.35)(6.56,-1.97) .. controls (4.17,-0.84) and (1.99,-0.18) .. (0,0) .. controls (1.99,0.18) and (4.17,0.84) .. (6.56,1.97)   ;
%Straight Lines [id:da697657458547721] 
\draw [color={rgb, 255:red, 128; green, 128; blue, 128 }  ,draw opacity=1 ][line width=0.75]    (390,35.68) -- (390,55.68) ;
%Straight Lines [id:da5669553098394792] 
\draw [color={rgb, 255:red, 128; green, 128; blue, 128 }  ,draw opacity=1 ][line width=0.75]    (288.69,70.27) -- (288.69,89.27) ;
%Curve Lines [id:da6392270085513881] 
\draw [color={rgb, 255:red, 128; green, 128; blue, 128 }  ,draw opacity=1 ][line width=0.75]    (266.8,31.62) .. controls (237.82,41.64) and (264.31,67.9) .. (289.3,70.75) ;
%Straight Lines [id:da9540651016424266] 
\draw [color={rgb, 255:red, 128; green, 128; blue, 128 }  ,draw opacity=1 ][line width=0.75]    (254,45) -- (254,64) ;
%Straight Lines [id:da4475450292219776] 
\draw [color={rgb, 255:red, 128; green, 128; blue, 128 }  ,draw opacity=1 ][line width=0.75]    (390,55.68) -- (405.3,49.2) ;
%Straight Lines [id:da3189705692065318] 
\draw [color={rgb, 255:red, 128; green, 128; blue, 128 }  ,draw opacity=1 ][line width=0.75]    (386.23,49.38) -- (405.3,49.2) ;
%Straight Lines [id:da6663763034332387] 
\draw [color={rgb, 255:red, 128; green, 128; blue, 128 }  ,draw opacity=1 ][line width=0.75]    (386.23,31.7) -- (386.23,50.7) ;
%Straight Lines [id:da036457188488025705] 
\draw [color={rgb, 255:red, 128; green, 128; blue, 128 }  ,draw opacity=1 ][line width=0.75]    (265.8,31.62) -- (265.8,50.62) ;
%Straight Lines [id:da43740047193846854] 
\draw  [dash pattern={on 2.25pt off 1.5pt}]  (384,167) -- (384,51.92) ;

% Text Node
\draw (409,8.4) node [anchor=north west][inner sep=0.75pt]  [font=\large,color={rgb, 255:red, 208; green, 2; blue, 27 }  ,opacity=1 ]  {$x$};
% Text Node
\draw (321,122.4) node [anchor=north west][inner sep=0.75pt]  [font=\footnotesize]  {$T_{x}^{( 0)}$};
% Text Node
\draw (211,89.4) node [anchor=north west][inner sep=0.75pt]  [font=\scriptsize]  {$2CD_{1}$};
% Text Node
\draw (297,1.4) node [anchor=north west][inner sep=0.75pt]  [font=\footnotesize]  {$2CD_{1} -D_{1}$};
% Text Node
\draw (419,34.4) node [anchor=north west][inner sep=0.75pt]  [font=\scriptsize]  {$D_{1}$};
% Text Node
\draw (388,1.4) node [anchor=north west][inner sep=0.75pt]  [font=\footnotesize]  {$D_{1}$};

\end{tikzpicture}
    \caption{$T_x^{(0)}$ for $n=2$.}
    \label{Fig.4}
\end{figure}
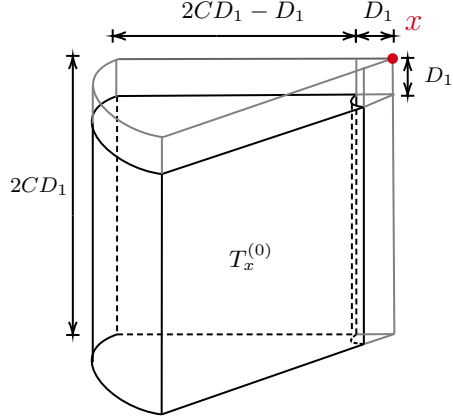

\begin{lemma}\label{lem:cover}
    $F(X)=\bigcup_{x\in H}T^{(0)}_{x}$.
\end{lemma}

\begin{proof}
    Let $y\in F(X)$. Apply the hitting property of $H$ with syndeticity $CD_1$ to the point $f_1^{D_1+1}\cdots f_n^{D_1+1}(y)$, and we get $\beta_1,\dots, \beta_n\in [0,CD_1]\cap\mathbb{N}$ such that $f_1^{\beta_1+D_1+1}\cdots f_n^{\beta_n+D_1+1}(y)=x\in H$. Thus $y\in T_x^{(0)}$.
\end{proof}

\begin{lemma}
     \itshape
    % Let $X$ be a standard Borel space. For $n\in \mathbb{N}$, suppose that $f_1,\cdots ,f_n:X\to X$ are countable-to-one commuting Borel functions. 
     Let $r\in\mathbb N^+$. Define a graph $G^r=(F(X),R)$ by
$$
xRy \quad\Longleftrightarrow\quad x\neq y\text{ and }\rho(x,y)\le r.
$$
Then there is a countable proper Borel coloring for $G^r$.
\end{lemma}

\begin{proof} Let $H_r\subseteq F(X)$ be a Borel $r$-forward-independent hitting set for $G_{f_1,\ldots,f_n}\upharpoonright F(X)$.

First consider arbitrary $\alpha_1,\ldots, \alpha_n\in\mathbb N$. Since each $f_i$ is
countable-to-one, the composition
$f_1^{\alpha_1}\cdots f_n^{\alpha_n}$
is countable-to-one. By the Luzin--Novikov theorem, there is a sequence of Borel functions
$\{P_m^{\alpha_1,\ldots, \alpha_n}\}_{m\in\mathbb N}$
such that, for every $z\in F(X)$,
$$
\left(f_1^{\alpha_1}\cdots f_n^{\alpha_n}\right)^{-1}(z)
=
\bigcup_{m\in \mathbb{N}}
P_m^{\alpha_1,\ldots, \alpha_n}(z).
$$

We now define a Borel coloring
$c:F(X)\to\mathbb N^{n+2}$. 
For $x\in F(X)$, let
$c(x)=(j_1,\ldots,j_n,k,\ell)$
where $(j_1+\cdots+j_n,j_1,\ldots,j_n)$ is the least tuple in the
lexicographic order of $\mathbb{N}^{n+1}$ such that
$f_1^{j_1+r}\cdots f_n^{j_n+r}(x)\in H_r$ (the existence of such $j_1,\ldots,j_n$ follows from the hitting property of
$H_r$ applied to
$f_1^r\cdots f_n^r(x)$),
and $k$ and $\ell$ are respectively the least integers such that
$f_1^{j_1}\cdots f_n^{j_n}(x)
\in
P_k^{r,\ldots,r}(H_r)$
and
$P_\ell^{j_1+r,\ldots,j_n+r}
\bigl(
f_1^{j_1+r}\cdots f_n^{j_n+r}(x)
\bigr)
=
x$.

%We first prove the following claim.

\begin{claim}
\label{clm:coloring-separation}
For every $k\in\mathbb N$, if $x\ne y$ and
$x,y\in P_k^{r,\ldots,r}(H_r)$,
then
$\rho(x,y)>r$.
\end{claim}

\noindent \emph{Proof of the claim}.
Toward a contradiction, assume that $\rho(x,y)\le r$. Then there exist
$a_1,\ldots,a_n$, $b_1,\ldots,b_n\in\mathbb N$
with
$a_1+\cdots+a_n+b_1+\cdots+b_n\le r$
such that
$f_1^{a_1}\cdots f_n^{a_n}(x)
=
f_1^{b_1}\cdots f_n^{b_n}(y)$.

Set
$x':=f_1^r\cdots f_n^r(x)$, 
$y':=f_1^r\cdots f_n^r(y)$.
Since $x\neq y$ and $x,y\in P_k^{r,\ldots,r}(H_r)$, we have $x'\neq y'$ and
$x',y'\in H_r$. For each $1\le i\le n$, define
$t_i:=\max\{a_i-b_i,0\}$ 
and
$t'_i:=\max\{b_i-a_i,0\}$.
Then
$t_i t'_i=0$
and
$t_i,t'_i\in [0,r]\cap \mathbb{N}$
for every $1\le i\le n$. A direct computation gives
$f_1^{t_1}\cdots f_n^{t_n}(x')
=
f_1^{t'_1}\cdots f_n^{t'_n}(y')$. However, this contradicts the $r$-forward-independence of $H_r$, since
$x'\ne y'$ are both in $H_r$.
\hfill $\mbox{\qed}_{\mbox{\scriptsize Claim}}$
\medskip

We now verify that $c$ is a proper coloring. For this, take distinct $x,y\in F(X)$ such that
$c(x)=c(y)=(j_1,\ldots,j_n,k,\ell)$.
By the definition of $c$, we have
$f_1^{j_1}\cdots f_n^{j_n}(x),f_1^{j_1}\cdots f_n^{j_n}(y)\in P_k^{r,\ldots,r}(H_r)$.
Moreover, $f_1^{j_1}\cdots f_n^{j_n}(x)\ne f_1^{j_1}\cdots f_n^{j_n}(y)$. Indeed, if $f_1^{j_1}\cdots f_n^{j_n}(x)=f_1^{j_1}\cdots f_n^{j_n}(y)$, then applying
$f_1^r\cdots f_n^r$ would give
$$
f_1^{j_1+r}\cdots f_n^{j_n+r}(x)
=
f_1^{j_1+r}\cdots f_n^{j_n+r}(y),
$$
and the definition of $\ell$ gives $x=y$,
a contradiction. By Claim~\ref{clm:coloring-separation},
$\rho(f_1^{j_1}\cdots f_n^{j_n}(x),f_1^{j_1}\cdots f_n^{j_n}(y))>r$.
Therefore,
$\rho(x,y)>r$.

Thus, $c$ is a proper Borel coloring of $G^r$ with countably many colors. This completes the proof.
\end{proof}

Continuing the proof of Theorem~\ref{thm:markers}, we fix a proper Borel coloring $c\colon F(X)\to\mathbb{N}$ of $G^{10^nCD_1}$.
For each $m\in \mathbb{N}$, let
$C_m:=c^{-1}(\{m\})\cap H$, then
$H=\bigsqcup_{m\in \mathbb{N}} C_m$.  For all $m\in \mathbb{N}$ and distinct $x,y \in C_m$, we have 
$\rho(x,y) > 10^nCD_1$.

Next, we adjust the family $\left\{T_x^{(0)} : x \in H\right\}$ in $2n$ steps.
The purpose of these adjustments is to separate the relevant $f_i$-faces by a uniform positive distance. 

% Before proceeding, we fix more notation and convention. For $A\subseteq F(X)$, distinct $g_1, \dots, g_k\in \{f_1,\dots, f_n\}$ and integers $\gamma_1,\dots, \gamma_k\in\mathbb{Z}$, let $g_1^{\gamma_1}\cdots g_k^{\gamma_k}(A)$ denote the set of $x\in F(X)$ such that there exist $y_0, y_1, \dots, y_k\in F(X)$ such that $y_0\in A$, $y_k=x$, and for any $1\leq j\leq k$, we have $g_j^{\gamma_j}(y_{j-1})=y_j$ if $\gamma_1\geq 0$ and $y_{j-1}=g_j^{|\gamma_j|}(y_j)$ if $\gamma_j<0$. This notation is consistent with the convention, when there is a single function $f_i$, that $f_i^{\gamma}(A)$ consists of the forward images of $A$ under $f_i^{\gamma}$ if $\gamma\in \mathbb{N}$, and $f_i^{\gamma}(A)$ consists of the preimages of $A$ under $f_i^{|\gamma|}$ if $\gamma\in\mathbb{Z}^-$. Because of the commutativity of $f_1,\dots, f_n$, the order in which these images or preimages are taken is not important.

Before proceeding, we fix some notation and conventions. Let $A\subseteq F(X)$ be a marker region, let
$g_1,\ldots,g_k$ be distinct members of
$\{f_1,\ldots,f_n\}$, and let
$\gamma_1,\ldots,\gamma_k\in\mathbb Z$. We define $g_1^{\gamma_1}\cdots g_k^{\gamma_k}(A)$ to be the set of all $x\in F(X)$ for which there exist
$y_0,y_1,\ldots,y_k\in F(X)$ such that
$y_0=x$, $y_k\in A$, and, for every $1\le j\le k$,
$y_{j-1}=g_j^{\gamma_j}(y_j)$ if $\gamma_j\ge0$,
and
$y_j=g_j^{|\gamma_j|}(y_{j-1})$ if $\gamma_j<0$.

If the exponents $\gamma_1,\ldots,\gamma_k$ are all nonnegative or
all nonpositive, the commutativity of
$f_1,\ldots,f_n$ implies that the resulting set is independent of
the order of the functions.
For mixed positive and negative exponents, this need not follow from
commutativity alone. The following claim gives a sufficient condition
for order independence.
\begin{claim}
\label{commutativity}
Suppose that 
$\gamma_j
<
\rho_{g_j}
\bigl(
\operatorname{root}(A),
F_{g_j}^{+}(A)
\bigr)$
whenever $\gamma_j>0$.
Then, for every permutation $\pi$ of $\{1,\ldots,k\}$,
$g_1^{\gamma_1}\cdots g_k^{\gamma_k}(A)=
g_{\pi(1)}^{\gamma_{\pi(1)}}\cdots
g_{\pi(k)}^{\gamma_{\pi(k)}}(A).$
\end{claim}

\noindent \emph{Proof of the claim}.
For $1\le i\le n$, let $\beta_i:=\rho_{f_i}\bigl(\operatorname{root}(A),F_i^+(A)\bigr)$ and define
$$
\delta_i=
\begin{cases}
\gamma_j,&\text{if }f_i=g_j\text{ for some }1\le j\le k,\\
0,&\text{otherwise}.
\end{cases}
$$

If $\delta_i\le0$, then
$\beta_i-\delta_i\ge0$.
If $\delta_i>0$, then $\delta_i=\gamma_j$ and $f_i=g_j$ for some
$j$, so the hypothesis guarantees $\delta_i<\beta_i$ and hence
$\beta_i-\delta_i\ge0$.
The set $g_{\pi(1)}^{\gamma_{\pi(1)}}\cdots
g_{\pi(k)}^{\gamma_{\pi(k)}}(A)$ consists of those $x\in F(X)$ for which there exist
$\alpha_1,\ldots,\alpha_n\in\mathbb N$ satisfying
$0\le\alpha_i\le \diam_{f_i}(A)$
for every $1\le i\le n$
and
\begin{align*}
\operatorname{root}(A)=
f_1^{\beta_1+\alpha_1-\delta_1}
\cdots
f_n^{\beta_n+\alpha_n-\delta_n}(x).
\end{align*}
This description is independent of the permutation $\pi$.
Consequently, 
$$
g_1^{\gamma_1}\cdots g_k^{\gamma_k}(A)
=
g_{\pi(1)}^{\gamma_{\pi(1)}}\cdots
g_{\pi(k)}^{\gamma_{\pi(k)}}(A).
$$
\vskip -18pt \hfill $\mbox{\qed}_{\mbox{\scriptsize Claim}}$

\bigskip

For each $x\in H$, $T_x^{(0)}$ is a rootless marker region with root $x$ satisfying $\rho_{f_j}
\bigl(
x,
F_j^+(T_x^{(0)})
\bigr)
>
\frac14D_1$
for $1\le j\le n$. Hence Claim~\ref{commutativity} applies to all exponents
$|\alpha_j|\le \frac14D_1$. We therefore define 
$$
\widetilde T_x^{(0)}=\bigcup _{  \alpha_1,\cdots , \alpha_n \in \left[-\frac{1}{4}D_1,\frac{1}{4}D_1\right]\cap\mathbb{Z}}f_1^{\alpha_1}\cdots  f_n^{\alpha_n}(T_x^{(0)}),
$$
and this definition is independent of
the order of the
functions.

 For each $x\in H$ and $1\leq i\leq n$, we will construct $T_x^{(i)}$ and $T_x^{(n+i)}$ below.  We also define the auxiliary regions
  $$
 \widetilde T_x^{(i)}=\bigcup _{  \alpha_1,\cdots, \alpha_{i-1}, \alpha_{i+1},\cdots , \alpha_n \in \left[-\frac{1}{4}D_1,\frac{1}{4}D_1\right]\cap\mathbb{Z}}f_1^{\alpha_1}\cdots f_{i-1}^{\alpha_{i-1}} f_{i+1}^{\alpha_{i+1}}\cdots f_n^{\alpha_n}(T_x^{(i)})
 $$
 and 
 $$
 \widetilde T_x^{(n+i)}=\bigcup _{\alpha_1,\cdots, \alpha_{i-1}, \alpha_{i+1},\cdots , \alpha_n \in \left[-\frac{1}{4}D_1,\frac{1}{4}D_1\right]\cap\mathbb{Z}}f_1^{\alpha_1}\cdots f_{i-1}^{\alpha_{i-1}} f_{i+1}^{\alpha_{i+1}}\cdots f_n^{\alpha_n}(T_x^{(n+i)}).
$$
The
sets $T_x^{(i)}$ and $T_x^{(n+i)}$ need not be marker regions. However, we will see that they can be viewed as unions of finitely many marker regions satisfying the condition of Claim~\ref{commutativity}, so the conclusion of Claim~\ref{commutativity} holds for them. Therefore, the definitions of
$\widetilde T_x^{(i)}$ and $\widetilde T_x^{(n+i)}$
are independent of the order of the functions.

\vspace{1em}
\noindent\textbf{Step 1.}
We will obtain $\{T_x^{(1)}:x\in H\}$ by extending each element of $\{T_x^{(0)}:x\in H\}$ in the backward $f_1$-direction so that
for all distinct $x,y\in H$:

\begin{enumerate}[label=(1.\arabic*),leftmargin=3.2em]
\item $T_x^{(0)}\subseteq T_x^{(1)}$;
\item $\rho_{f_1}\!\bigl(F_1^-(T_x^{(0)}),F_1^-(T_x^{(1)})\bigr)\le \frac{1}{10n}D_1$;
\item For any $z\in F_1^-(T_x^{(1)})$, if $\rho(z, F_1^-(\widetilde T_y^{(1)}))=\rho_{f_1}(z, F_1^-(\widetilde T_y^{(1)}))$, then $\rho_{f_1}(z, F_1^-(\widetilde T_y^{(1)}))\ge D'$ or $\rho_{f_1}(z, F_1^-(\widetilde T_y^{(1)}))=0$.
\end{enumerate}

Define $T_x^{(1)}$ by induction on $m\in\mathbb{N}$ for $x\in C_m$. 
Suppose $m\geq 0$ and $x\in C_m$. Assume that $T_y^{(1)}$ and $\widetilde T_y^{(1)}$ have been defined for all $y\in\bigcup_{m'<m}C_{m'}$. 
We make the following adjustment to $T_x^{(0)}$. Define 
$$
\begin{aligned}
A_0=\Big\{z\in F_1^-(T_x^{(0)})\colon  
& \rho(z, F_1^-(\widetilde T_y^{(1)}))=
\rho_{f_1}(z, F_1^-(\widetilde T_y^{(1)}))<D' \text{ for some } y \in \bigcup_{m'<m}C_{m'}, \\
&\mbox{or } \rho(z, F_1^-(\widetilde T_y^{(0)}))=
\rho_{f_1}(z, F_1^-(\widetilde T_y^{(0)}))<D' \text{ for some } y \in \bigcup_{m'>m}C_{m'}\Big\}.
\end{aligned}
$$

We remark that, in the above definition, for
$y\in\bigcup_{m'>m}C_{m'}$, since $\widetilde T_y^{(1)}$ is not yet defined, we consider the initial region
$\widetilde T_y^{(0)}$ instead. This is important when $f_1$ is not surjective: a backward
face of $\widetilde T_y^{(0)}$ may contain defects due to the absence of $f_1$-preimages
and therefore may not be literally separable from other faces by a later backward extension. Such faces
are therefore treated as fixed obstacles from the beginning.

For $t\ge0$, define inductively 
$$
\begin{aligned}
    A_{t+1}=\Big\{z\in f_1^{-1}(A_t)\colon  &\rho(z, F_1^-(\widetilde T_y^{(1)}))=
\rho_{f_1}(z, F_1^-(\widetilde T_y^{(1)}))<D' \text{ for some } y \in \bigcup_{m'<m}C_{m'} \\
&\mbox{or } \rho(z, F_1^-(\widetilde T_y^{(0)}))=
\rho_{f_1}(z, F_1^-(\widetilde T_y^{(0)}))<D' \text{ for some } y \in \bigcup_{m'>m}C_{m'}\Big\}.
\end{aligned}
$$

\begin{claim}
\label{counting}
Let $s:=2D'(4C+1)^{n-1}$. Then
$A_s=\varnothing$.
\end{claim}

\noindent \emph{Proof of the claim}. 
Toward a contradiction, assume that $A_s\neq\varnothing$. Fix
$z_s\in A_s$. Since $A_{t+1}\subseteq f_1^{-1}(A_t)$ for all $t\ge 0$, we have points $z_t\in A_t$ for $0\le t\le s$, such that
$f_1(z_{t+1})=z_t$ for each $0\le t<s$.
Note that 
\begin{align}
z_t=f_1^{\,t'-t}(z_{t'})
\qquad\text{whenever }0\le t<t'\le s.
\tag{1}
\end{align}

For each $0\le t<s$, by the definition of $A_t$, either there exists
$y_t\in \bigcup_{m'<m}C_{m'}$
such that
\begin{align}
   \rho\bigl(z_t,F_1^-(\widetilde T_{y_t}^{(1)})\bigr)
=
\rho_{f_1}\bigl(z_t,F_1^-(\widetilde T_{y_t}^{(1)})\bigr)<D',
\tag{2} 
\end{align}

or there exists
$y_t\in \bigcup_{m'>m}C_{m'}$
such that
\begin{align}
    \rho\bigl(z_t,F_1^-(\widetilde T_{y_t}^{(0)})\bigr)
=
\rho_{f_1}\bigl(z_t,F_1^-(\widetilde T_{y_t}^{(0)})\bigr)<D'.
\tag{3}
\end{align}

Choose
$v_t\in F_1^-(\widetilde T_{y_t}^{(1)})$
witnessing (2) or $v_t\in F_1^-(\widetilde T_{y_t}^{(0)})$ witnessing (3). Since the graph metric values in (2) and (3) are realized purely
in the $f_1$-direction, there exist
$\alpha_{1,t}, \beta_{1,t}\in\mathbb N$
such that
\begin{align}
    f_1^{\alpha_{1,t}}(z_t)=f_1^{\beta_{1,t}}(v_t),
\qquad
\alpha_{1,t}\beta_{1,t}=0,
\qquad
|\alpha_{1,t}-\beta_{1,t}|<D'.
\tag{4}
\end{align}

We first bound the number of distinct witnesses $y_t$. Since
$z_0=f_1^t(z_t)\in A_0\subseteq F_1^-(T_x^{(0)})$,
there exist
$a_{2,t},\ldots,a_{n,t}\in (D_1,2CD_1]\cap \mathbb{N}$
such that
$$
f_1^{2CD_1+t} f_2^{a_{2,t}}\cdots f_n^{a_{n,t}}(z_t)=x.
$$
Similarly, from $v_t\in F_1^-(\widetilde T_{y_t}^{(1)})$ or $v_t\in F_1^-(\widetilde T_{y_t}^{(0)})$, using condition \textup{(1.2)}, there exist
$b_{2,t},\ldots,b_{n,t}\in (0,2CD_1+D_1)\cap\mathbb N$
and $r_t\ge 2CD_1$ such that
$$
f_1^{r_t} f_2^{b_{2,t}}\cdots f_n^{b_{n,t}}(v_t)=y_t.
$$
For $2\le i\le n$, set
$$
\Delta_i(y_t):=a_{i,t}-b_{i,t}.
$$
%The definition is independent of the choices in $(4)$ and $(5)$, since we work on the free part $F(X)$. Moreover,
Then we have $\Delta_i(y_t)\in [-2CD_1,\,2CD_1]\cap \mathbb{Z}$.

Define
$$
q_i(y_t):=\left\lfloor\frac{\Delta_i(y_t)}{D_1}\right\rfloor
\in [-2C, 2C]\cap \mathbb{Z},
\qquad 2\le i\le n,
$$
and put
$$
\sigma(y_t):=(q_2(y_t),\ldots,q_n(y_t)).
$$

We claim that the map
$y_t\longmapsto \sigma(y_t)$
is injective on the set of witnesses appearing in (2) and (3). Suppose that
$\sigma(y_t)=\sigma(y_{t'})$
for some $t<t'$. Then, for each $2\le i\le n$, the values
$\Delta_i(y_t)$ and $\Delta_i(y_{t'})$ lie in the same interval of length
$D_1$. Hence
$$
\rho_{f_i}(y_t,y_{t'})<D_1
\qquad\text{for every }2\le i\le n.
$$

We also compare the $f_1$-direction. By $(1)$ and $(4)$,
$f_1^{\alpha_{1,t}+t'-t}(z_{t'})=f_1^{\beta_{1,t}}(v_t)$.
Therefore
$$
f_1^{D_1}(z_{t'})
=
f_1^{D_1+\beta_{1,t}-\alpha_{1,t}-(t'-t)}(v_t).
$$
Since
$t'-t\le s$ and $|\alpha_{1,t}-\beta_{1,t}|<D'$, we have 
$f_1^{D_1}(z_{t'})\in \widetilde T_{y_t}^{(0)}$ in both cases of (2) and (3). By the definition of $y_{t'}$, we also have
$f_1^{D_1}(z_{t'})\in \widetilde T_{y_{t'}}^{(0)}$. Thus, there exist a point $u$ and
integers $h_i,h'_i\in\mathbb N$, $1\le i\le n$, such that
$$
f_1^{h_1}\cdots f_n^{h_n}(u)=y_t,
\qquad
f_1^{h'_1}\cdots f_n^{h'_n}(u)=y_{t'},
$$
and
$$
0\le h_i,h'_i\leq D_1,
\qquad
h_ih'_i=0
\qquad\text{for every }1\le i\le n.
$$
Hence, by commutativity,
$f_1^{h'_1}\cdots f_n^{h'_n}(y_t)
=
f_1^{h_1}\cdots f_n^{h_n}(y_{t'})$.
Since $H$ is $D_1$-forward-independent and $y_t,y_{t'}\in H$, we must have
$y_t=y_{t'}$. 

It follows from the injectivity of $y_t\mapsto \sigma(y_t)$ that the
number of distinct witnesses $y_t$ is at most 
\begin{align}
    \bigl|\,([-2C, 2C]\cap \mathbb{Z})^{n-1}\,\bigr|=(4C+1)^{n-1}.
\tag{5}
\end{align}

It remains to bound how often a fixed witness can occur. Fix $y\in H$. Suppose
that
$y_t=y_{t'}=y$ with $t<t'$. First suppose $y\in \bigcup_{m'<m}C_{m'}$. Then by $(2)$, both $z_t$ and $z_{t'}$ are at $f_1$-pseudo-distance $<D'$
from $F_1^-(\widetilde T_y^{(1)})$. Since
$z_t=f_1^{\,t'-t}(z_{t'})$,
we have
$\rho_{f_1}(z_t,z_{t'})=t'-t.$
If $t'-t\ge 2D'$, then the triangle inequality for $\rho_{f_1}$ gives
$$
\rho_{f_1}\bigl(z_{t'},F_1^-(\widetilde T_y^{(1)})\bigr)
\ge
\rho_{f_1}(z_t,z_{t'})
-
\rho_{f_1}\bigl(z_t,F_1^-(\widetilde T_y^{(1)})\bigr)
>
2D'-D'=D',
$$
contradicting $(2)$. Hence
$t'-t<2D'$. Now suppose $y\in \bigcup_{m'>m}C_{m'}$. Then we also have $t'-t<2D'$ from a similar argument. 
Thus each fixed witness $y$ occurs for at most $2D'-1$ values of $t$.

Combining this with $(5)$, we get
$$
\bigl|\{0\leq t<s\colon A_t\neq\varnothing\}\bigr|
\le
(2D'-1)(4C+1)^{n-1}
<
2D'(4C+1)^{n-1}
=s.
$$
This contradicts the existence of the chain
$z_0,z_1,\ldots,z_s$.
Therefore $A_s=\varnothing$.
\hfill $\mbox{\qed}_{\mbox{\scriptsize Claim}}$
\medskip

We then define $T_x^{(1)}=T_x^{(0)}\cup \bigcup_{1\le t< s}A_t$.
Since $s=2D'(4C+1)^{n-1}$,
we may take $D_1$ so large that
$\frac{1}{10n}D_1>s$, then \textup{(1.2)} holds. The construction of the sets $A_t$
enforces \textup{(1.3)}. Since $\rho(x,y)>10^nCD_1$ for distinct $x, y\in C_m$, the constructions of $T_x^{(1)}$ and $T_y^{(1)}$ can take place simultaneously without interfering with each other. 

Each $A_t$, and hence $T_x^{(1)}$, is a union of marker regions.
Indeed, fix a witness $y$, an integer $\beta<D'$, and one of the
two $f_1$-directions. By the induction on $m$,
$\widetilde T_y^{(\ell)}$ is a union of marker regions, where
$\ell=1$ if $y\in\bigcup_{m'<m}C_{m'}$ and $\ell=0$ otherwise.
The corresponding part of $A_0$ is therefore a union of nonempty
sets of the form
$$
F_1^-(T_x^{(0)})
\cap
f_1^{\pm\beta}
\bigl(F_1^-(B)\bigr),
$$
where $B$ is a marker region contained in $\widetilde T_y^{(\ell)}$, and each such set is again a marker
region. Thus, $A_0$ is a union of marker regions.

If $A_t$ is a union of marker regions, then so is
$f_1^{-1}(A_t)$. Thus, we partition $f_1^{-1}(A_t)$ into several marker regions. Applying a similar argument to each single marker region shows that $A_{t+1}$ is also a union of
marker regions. Therefore, the assertion follows by induction on $t$.

The bounds on the face
adjustments ensure that every marker region $B$ in these
decompositions satisfies
$\rho_{f_j}
\bigl(
\operatorname{root}(B),
F_j^+(B)
\bigr)
>
\frac14D_1$
for $j\ne 1$.
Hence Claim~\ref{commutativity} applies to all exponents
$|\alpha_j|\le \frac14D_1$. Therefore, the definition of
$\widetilde T_x^{(1)}$
is independent of the order of the functions.

Figure~\ref{fig:step1-adjustment} illustrates the local effect of
Step~1 in the case $n=2$, where $f_1$ is injective and $f_2$ is
countable-to-one. The left panel shows the initial region
$T_x^{(0)}$, while the right panel shows the adjusted region
$T_x^{(1)}$. In the right panel, the blue and orange parts represent
$\bigcup_{1\le t<s} A_t$.

\begin{figure}[htbp]
    \centering
    \input{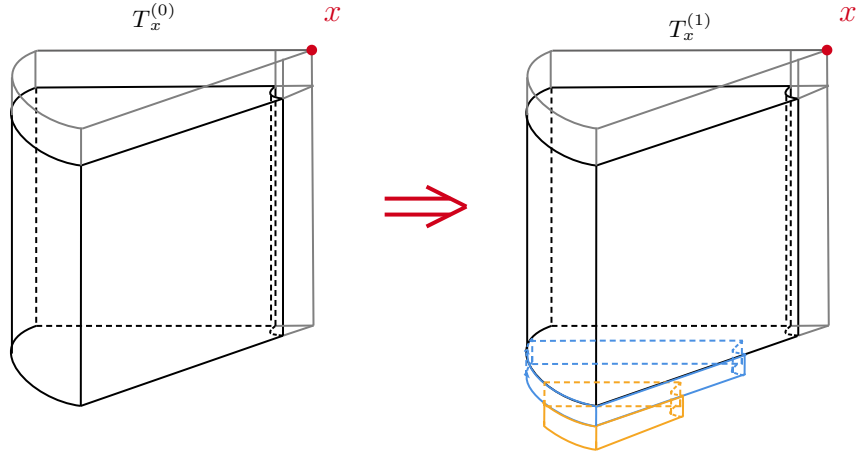}
    \caption{The local adjustment from $T_x^{(0)}$ to $T_x^{(1)}$
    in Step~1.}
    \label{fig:step1-adjustment}
\end{figure}

%----------------------------
\medskip
\noindent\textbf{Steps $2, \dots, n$.}
For each $i=2,\ldots,n$, we obtain
$\{T_x^{(i)}:x\in H\}$ from
$\{T_x^{(i-1)}:x\in H\}$ by repeating Step~1 with $f_1$ and
$F_1^-$ replaced by $f_i$ and $F_i^-$, respectively.
The same argument shows that each $T_x^{(i)}$ is a union of marker
regions to which Claim~\ref{commutativity} applies. Therefore, the definition of
$\widetilde T_x^{(i)}$
is independent of the order of the functions.

More precisely, for all distinct $x,y\in H$, the following
properties hold:
\begin{enumerate}[label=($i$.\arabic*),leftmargin=3.2em]
  \item $T_x^{(i-1)}\subseteq T_x^{(i)}$;
  \item $\rho_{f_i}\!\bigl(F_i^-(T_x^{(i-1)}),\,F_i^-(T_x^{(i)})\bigr)\le \frac{1}{10n}D_1$;
  \item For any $z\in F_i^-(T_x^{(i)})$, if $\rho(z, F_i^-(\widetilde T_y^{(i)}))=\rho_{f_i}(z, F_i^-(\widetilde T_y^{(i)}))$, then $\rho_{f_i}(z, F_i^-(\widetilde T_y^{(i)}))\ge D'$ or $\rho_{f_i}(z, F_i^-(\widetilde T_y^{(i)}))=0$.
\end{enumerate}

\medskip
\noindent
At the end of these steps, we obtain the family $\{T_x^{(n)}\colon x\in H\}$. Next, we perform a round of adjustments for $F_i^+$-faces.

%-----------------
\medskip
\noindent\textbf{Step $n+1$.}
In this step we obtain $\{T_x^{(n+1)}:x\in H\}$ by extending each element of $\{T_x^{(n)}:x\in H\}$ in the $f_1$-direction so that
for all distinct $x,y\in H$:
\begin{quote}
\begin{enumerate}[label=($n+1$.\arabic*),leftmargin=3.2em]
\item $T_x^{(n)}\subseteq T_x^{(n+1)}$;
  \item $\rho_{f_1}\bigl(F_1^+(T_x^{(n)}),F_1^+(T_x^{(n+1)})\bigr)\le \frac{1}{10n}D_1$;
\item  For any $z\in F_1^+(T_x^{(n+1)})$, if $\rho(z, F_1^+(\widetilde T_y^{(n+1)}))=
\rho_{f_1}(z, F_1^+(\widetilde T_y^{(n+1)}))$, \\then $\rho_{f_1}(z, F_1^+(\widetilde T_y^{(n+1)}))\ge D'$;
\item  For any $z\in F_1^+(T_x^{(n+1)})$, if $\rho(z, F_1^-(\widetilde T_y^{(n+1)}))=
\rho_{f_1}(z, F_1^-(\widetilde T_y^{(n+1)}))$, \\then $\rho_{f_1}(z, F_1^-(\widetilde T_y^{(n+1)}))\ge D'$.
\end{enumerate}
\end{quote}

Exactly as we did in the previous steps, define $T_x^{(n+1)}$ by induction on $m$ for $x\in C_m$. Fix $m\in \mathbb{N}$ and $x\in C_m$. Assume that $T_y^{(n+1)}$ and $\widetilde T_y^{(n+1)}$ have been defined for all $y\in\bigcup_{m'<m}C_{m'}$.
We define $A_0'\subseteq F_1^+(T_x^{(n)})$ to be the points which violate $(n+1.3)$ or $(n+1.4)$. More precisely, for each $x\in H$, define 
$$
\begin{aligned}
    A'_0=\Big\{z\in F_1^+(T_x^{(n)})\colon &\rho(z, F_1^+(\widetilde T_y^{(n+1)}))=
\rho_{f_1}(z, F_1^+(\widetilde T_y^{(n+1)}))< D' \text{ for some } y \in \bigcup_{m'<m}C_{m'}, \\& \mbox{or } \rho(z, F_1^-(\widetilde T_y^{(n)}))=
\rho_{f_1}(z, F_1^-(\widetilde T_y^{(n)}))< D'\text{ for some } y \in H\Big\}.
\end{aligned}
$$

Here the second condition ranges over all $y\in H$, since the
$F_1^-$-faces were fixed during the first $n$ adjustment stages.

Inductively, for $t\ge 0$, define 
 $$
 \begin{aligned}
     A_{t+1}'=\Big\{z\in f_1(A'_t)\colon &\rho(z, F_1^+(\widetilde T_y^{(n+1)}))=
\rho_{f_1}(z, F_1^+(\widetilde T_y^{(n+1)}))< D' \text{ for some } y \in \bigcup_{m'<m}C_{m'},\\& \mbox{or } \rho(z, F_1^-(\widetilde T_y^{(n)}))=
\rho_{f_1}(z, F_1^-(\widetilde T_y^{(n)}))< D'\text{ for some } y \in H\Big\}.
 \end{aligned}
$$
An argument analogous to the proof of
Claim~\ref{counting} shows that $A'_{s'}=\varnothing$ for $s'=4D'(4C+3)^{n-1}$.

We now set $T_x^{(n+1)}:=T_x^{(n)}\cup \bigcup_{1\le t<s'}A'_t$. Choosing $D_1$ so large that
$\frac{1}{10n}D_1>s'$, we obtain $(n+1.2)$ holds. The construction of the sets $A'_t$
enforces $(n+1.3)$ and $(n+1.4)$. 

Similarly, $T_x^{(n+1)}$ is a union of marker regions and the bounds on the face
adjustments ensure that every marker region $B$ in these
decompositions satisfies
$\rho_{f_j}
\bigl(
\{\operatorname{root}(B)\},
F_j^+(B)
\bigr)
>
\frac14D_1$
for $j\ne 1$.
Hence Claim~\ref{commutativity} applies to all exponents
$|\alpha_j|\le \frac14D_1$. Therefore, the definition of
$\widetilde T_x^{(n+1)}$
is independent of the order of the functions.

%------------------------------
\medskip
\noindent\textbf{Steps $n+2,\dots,2n$.}
For each $i=2,\ldots,n$, we obtain
$\{T_x^{(n+i)}:x\in H\}$ from
$\{T_x^{(n+i-1)}:x\in H\}$ by repeating Step~$n+1$ with
$f_1$, $F_1^+$, and $F_1^-$ replaced by
$f_i$, $F_i^+$, and $F_i^-$, respectively.
The same argument shows that each $T_x^{(n+i)}$ is a union of marker
regions to which Claim~\ref{commutativity} applies. Therefore, the definition of
$\widetilde T_x^{(1)}$
is independent of the order of the functions.

More precisely, for all distinct $x,y\in H$, the following
properties hold:
\begin{quote}
\begin{enumerate}[label=($n+i$.\arabic*),leftmargin=3.2em]
\item $T_x^{(n+i-1)}\subseteq T_x^{(n+i)}$;
\item $\rho_{f_i}\bigl(F_i^+(T_x^{(n+i-1)}),F_i^+(T_x^{(n+i)})\bigr)\le \frac{1}{10n}D_1$;
\item  For any $z\in F_i^+(T_x^{(n+i)})$, if $\rho(z, F_i^+(\widetilde T_y^{(n+i)}))=
\rho_{f_i}(z, F_i^+(\widetilde T_y^{(n+i)}))$, \\then $\rho_{f_i}(z, F_i^+(\widetilde T_y^{(n+i)}))\ge D'$ ;
\item  For any $z\in F_i^+(T_x^{(n+i)})$, if $\rho(z, F_i^-(\widetilde T_y^{(n+i)}))=
\rho_{f_i}(z, F_i^-(\widetilde T_y^{(n+i)}))$, \\then $\rho_{f_i}(z, F_i^-(\widetilde T_y^{(n+i)}))\ge D'$ .
\end{enumerate}
\end{quote}
\medskip

At the end of these steps, we obtain the family $\{T_x^{(2n)}:x\in H\}$. 
Their faces are uniformly separated.

% \begin{lemma}
% \label{lem:face-separation}
% For each $1\le i\le n$, the family $\{T_x^{(2n)}:x\in H\}$ satisfies
% $$
% \rho_{f_i}\bigl(F_i(T_x^{(2n)}),F_i(T_y^{(2n)})\bigr)\ge D'
% $$
% whenever $x,y\in H$ are distinct and
% $\widetilde T_x^{(2n)}\cap\widetilde T_y^{(2n)}\neq\varnothing$.
% \end{lemma}

% \begin{proof}
% Fix $1\le i\le n$ and distinct $x,y\in H$. $\rho_{f_i}\bigl(F_i^-(T_x^{(2n)}),F_i^-(T_y^{(2n)})\bigr)\ge D'$ follows
% from the condition imposed in the $i$-th $F_i^-$-face adjustment, namely
% $(i.3)$. $\rho_{f_i}\bigl(F_i^+(T_x^{(2n)}),F_i^+(T_y^{(2n)})\bigr)\ge D'$ and $\rho_{f_i}\bigl(F_i^+(T_x^{(2n)}),F_i^-(T_y^{(2n)})\bigr)\ge D'$
% follow from the $i$-th $F_i^+$-face adjustment, namely
% \textup{($n+i$.3)} and \textup{($n+i$.4)}. Combining these four cases gives $\rho_{f_i}\bigl(F_i(T_x^{(2n)}),F_i(T_y^{(2n)})\bigr)\ge D'$. 
% \end{proof}

%--------------------------
%After the above adjustments, we obtain the following face-separation lemma.

\subsection{Rootless marker regions}

We now use the face-separated family $\{T_x^{(2n)}:x\in H\}$ to construct Borel marker regions of controlled shape. The first outcome is a decomposition into rootless regions.

\begin{proof}[Proof of Theorem~\ref{thm:markers}]
Define a subequivalence relation $R_D^{0}\subseteq E_{f_1,\cdots,f_n}$ by
$$
\begin{aligned}
x\,R_D^{0}\,y
\iff\ & x\,E_{f_1,\cdots,f_n}\,y
\mbox{ and }
\forall z\in H\cap [x]_{E_{f_1,\cdots,f_n}}\,
\bigl(x\in T_z^{(2n)} \iff y\in T_z^{(2n)}\bigr).
\end{aligned}
$$
By the Luzin--Novikov theorem, the universal quantifier in the above formula can be replaced by a universal quantifier over natural numbers. Thus $R_D^{0}$ is Borel.

Geometrically, an $R_D^{0}$-class may be viewed as the region obtained by superimposing the regions $\{T_z^{(2n)}:z\in H\}$ and cutting along overlaps. In general, such a superposition may produce irregular classes whose boundary faces do not yet have the form of a single rootless region.

Figure~\ref{Fig.5} illustrates a typical example for $n=2$, where $f_1$ is injective and $f_2$ is countable-to-one. The green, purple, black, and blue
regions represent $T_{z_1}^{(4)}$, $T_{z_2}^{(4)}$, $T_{z_3}^{(4)}$, and
$T_{z_4}^{(4)}$, respectively. The upper-left panel shows their projection in
the $f_2$-direction, while the middle panel shows the overlap pattern. The
red parts indicate the overlap regions.

\begin{figure}[H]
    \centering
     \includegraphics[trim=1cm 0.8cm 4cm 1cm,clip,width=1\linewidth]{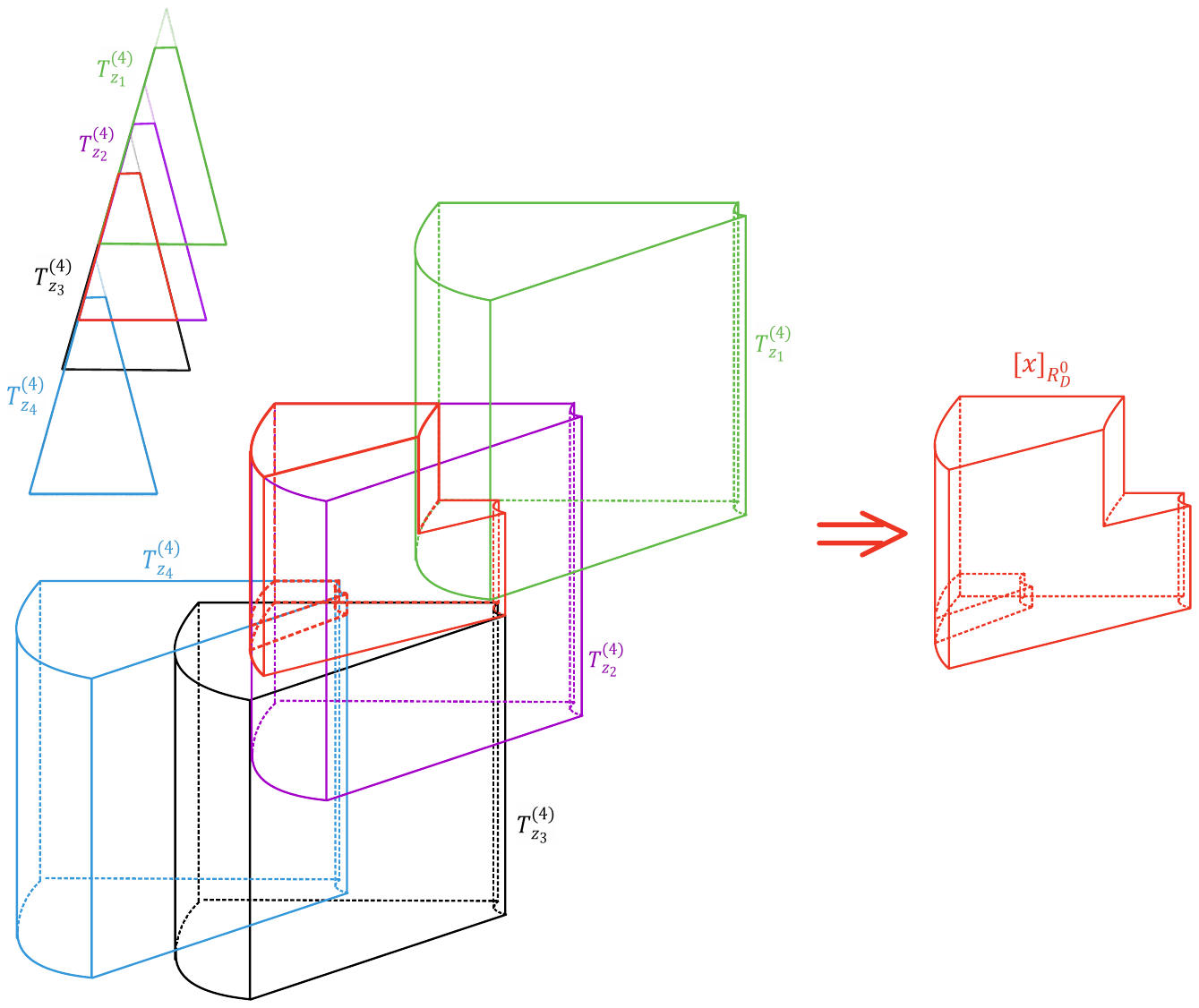}
    \caption{Superposition of overlapping $T_z^{(2n)}$ regions.}
    \label{Fig.5}
\end{figure}

Note that by the construction of 
$\{T_z^{(2n)}\}_{z\in H}$, if $x,y\in F(X)$ and $$f_1^{\frac14D_1}\cdots f_n^{\frac14D_1}(x)
=
f_1^{\frac14D_1}\cdots f_n^{\frac14D_1}(y),$$
then
$x R_D^0 y$.
The next claim isolates a uniform local configuration inside every $R_D^{0}$-class, which will be used to perform the final adjustment.

\begin{claim}
\label{clm:local-segment}
For
$y\in F(X)$, there are some $x$ and $0\le\beta\le D$ such that
$f_i^\beta(x)=y$
and
$\bigcup_{0\le\alpha\le D}
f_i^{-\alpha}(y)
\subseteq [x]_{R_D^0}$.
\end{claim}

\noindent\emph{Proof of the claim}.
Toward a contradiction, assume that there exist $x\in F(X)$ and
$1\le i\le n$ such that, for every $0\le\beta\le D$, there are
$0\le\alpha\le D$ and
$u\in f_i^{-\alpha}\bigl(\{f_i^\beta(x)\}\bigr)$
with $u\notin[x]_{R_D^0}$.

Taking $\beta=0$, choose $u\in f_i^{-\alpha}(x)$ for some
$1\le\alpha\le D$ such that $u\notin[x]_{R_D^0}$. Since
$f_i^\alpha(u)=x\in[x]_{R_D^0}$, there is a least
$\alpha_0<\alpha$ such that
$$
f_i^{\alpha_0}(u)\notin[x]_{R_D^0}
\qquad\text{and}\qquad
f_i^{\alpha_0+1}(u)\in[x]_{R_D^0}.
$$
By the definition of $R_D^0$, these two points are distinguished by
$T_z^{(2n)}$ for some
$z\in H\cap[x]_{E_{f_1,\ldots,f_n}}$. Since they are consecutive
in the $f_i$-direction, one of them belongs to
$F_i(T_z^{(2n)})$. Without loss of generality, assume that
$f_i^{\alpha_0}(u)\in F_i(T_z^{(2n)})$.

Set $\eta:=D-\alpha+\alpha_0+1$ and 
$w:=f_i^\eta(x)$.
Then $1\le\eta\le D$ and
$f_i^{D+1}\bigl(f_i^{\alpha_0}(u)\bigr)=w$.

We claim that no point in
$\bigcup_{0\le\gamma\le D}f_i^{-\gamma}(w)$
belongs to $F_i(T_{z'}^{(2n)})$ for any
$z'\in H\cap[x]_{E_{f_1,\ldots,f_n}}$. Otherwise, if
$v\in f_i^{-\gamma}(w)\cap F_i(T_{z'}^{(2n)})$
for some $0\le\gamma\le D$, then there exists $v'\in F(X)$ such that $f_1^{D'}\cdots f_n^{D'}(v')
=
f_1^{D'}\cdots f_n^{D'}(v)$ and $f_i^{D+1-\gamma}(f_i^{\alpha_0}(u))=v'.$ It follows that $\rho
\bigl(
f_i^{\alpha_0}(u),v'
\bigr)=\rho_{f_i}
\bigl(
f_i^{\alpha_0}(u),v'
\bigr)
\le D+1<D'$. 
This contradicts either condition \textup{($i$.3)} imposed in the
$i$-th $F_i^-$-face adjustment, or one of the conditions
\textup{($n+i$.3)} and \textup{($n+i$.4)} imposed in the
$i$-th $F_i^+$-face adjustment.
Hence, for each $z'\in H\cap [x]_{E_{f_1,\cdots,f_n}}$, either all points in $\bigcup_{0\le\gamma\le D}f_i^{-\gamma}(w)$
belong to $T_{z'}^{(2n)}$, or none of them do.
Since $f_i^{\alpha_0+1}(u)\in f_i^{-D}(w)$
belongs to this set and lies in $[x]_{R_D^0}$, it follows that
$\bigcup_{0\le\gamma\le D}f_i^{-\gamma}(w)\subseteq [x]_{R_D^0}$,
contradicting our assumption.
\hfill$\mbox{\qed}_{\mbox{\scriptsize Claim}}$

\medskip
We are now ready to define a sequence of equivalence relations $R_D^1, \dots, R_D^n$ and will set $R_D=R_D^n$. For each $1\leq i\leq n$, $R_D^i$ is a refinement of $R_D^{i-1}$ and each $R_D^i$-class is obtained from an $R_D^{i-1}$-class by a number of cuts. For non-surjective cases, we can add extra points to make each region in $F(X)$ have a full preimage, which ensures us to complete the following local cuts. Therefore, we will assume that all functions are surjections in the following proof.

We use the following notation in the inductive definition of $R_D^i$.
For $x\in F(X)$ define its \emph{$i$-th level} to be the least integer $L^i_x\geq 0$ such that
$$
f_i^{L^i_x}(x)\in [x]_{R_D^0}
\quad\text{and}\quad
f_i^{L^i_x+1}(x)\notin [x]_{R_D^0}.
$$
For $x\in F(X)$, define the associated \emph{$i$-th level set} to be 
$$
\mathrm{Lev}_i(x):=\{z\in [x]_{R_D^0}:L^i_z=L^i_x \mbox{ and } \rho_{f_i}(x,z)=0\}.
$$

Now, for $1\le i \le n$, suppose that $R_D^{i-1}$ has been defined and satisfies the conclusion in Claim~\ref{clm:local-segment}. We construct $R_D^i$ as follows. Fix $x\in F(X)$ and let $A=[x]_{R^{i-1}_D}$. First, for $y\in F_i^{-}(A)$, set
$$
[y]_{R_D^{i}}
:=
\bigcup_{0\le j\le D+1+\bigl(L^i_x \bmod (D+1)\bigr)}
f_i^{j}\bigl(\mathrm{Lev}_i(y)\bigr)\cap A.
$$
Then, for $z\in A\setminus \bigcup_{y\in F_i^-(A)}[y]_{R_D^i}$ and $L^i_z\bmod (D+1)=0$, we define
$$
[z]_{R_D^i}=\bigcup_{0\le j\le D}
f_i^{\,j}\bigl(\mathrm{Lev}_i(z)\bigr)\setminus \bigcup_{y\in F_i^-(A)}[y]_{R_D^i}.
$$
This completes the definition of $R_D^i$ as a subequivalence relation of $R_D^{i-1}$.
It is clear from the construction that the conclusion of Claim~\ref{clm:local-segment} still holds with $R_D^i$ replacing $R_D^0$.

Figure~\ref{Fig.6} illustrates the construction for $n=2$, where $f_1$ is injective and $f_2$ is countable-to-one. Here the class $[x]_{R_D^0}$ is refined by two successive families of cuts: first by blue lines parallel to the $f_2$-direction, and then by orange lines parallel to the $f_1$-direction.

\begin{figure}[H]
    \centering
     \includegraphics[trim=2cm 18cm 2cm 1.5cm,clip,width=1\linewidth]{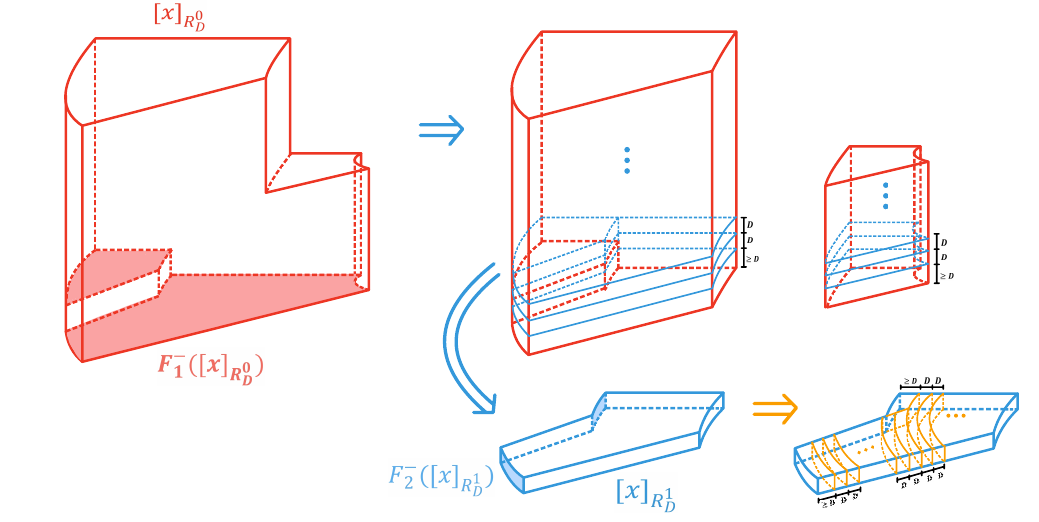}
    \caption{A two-step refinement of $[x]_{R_D^0}$.}
    \label{Fig.6}
\end{figure}

%Consequently, for all $1 \le m\leq n$, $R_D^{m}$ is a Borel subequivalence relation of $R_D^{m-1}$, and every $R_D^{m}$-class has $f_i$-edge length at least $D$ for $1\le i\leq m$ by Claim~\ref{clm:local-segment} and our constructions. It follows from the construction that, for every $R_D^n$-class $B$, every $1\le i\le n$, and all $y,z\in F_i(B)$, either $\rho_{f_i}(y,z)=0$ or $\rho_{f_i}(y,z)\ge D$.

Now set $R_D:=R_D^n$. By induction, the conclusion of
Claim~\ref{clm:local-segment} remains valid with $R_D$ in place
of $R_D^0$. And through our construction, every $R_D$-class is a rootless
region satisfying
$d_i([x]_{R_D})\ge D$
for every $1\le i\le n$.

We next verify that $R_D$ is smooth. Since $G_{f_1, \dots, f_n}\upharpoonright F(X)=(F(X), R)$ is a locally countable Borel graph, it has countable Borel edge chromatic number (\cite[Proposition 4.10]{KechrisSoleckiTodorcevic1999}). Let $c\colon R\to \mathbb{N}$ be a proper Borel edge coloring of $G_{f_1, \dots, f_n}\upharpoonright F(X)$. For each $R_D$-class $A$ and for every $x\in A$, let $\alpha_1(x), \dots, \alpha_n(x)\in \mathbb{N}$ be the unique numbers satisfying 
$$ f_1^{\alpha_1(x)}\cdots f_n^{\alpha_n(x)}(x)=\mbox{root}(A). $$
Let $p(x)$ be the unique path from $x$ to $\mbox{root}(A)$ of length $\ell=\alpha_1(x)+\cdots+\alpha_n(x)$ consisting of vertices
$$ x, f_1(x), \dots, f_1^{\alpha_1(x)}(x), f_1^{\alpha_1(x)}f_2(x), \dots, f_1^{\alpha_1(x)}\cdots f_n^{\alpha_n(x)}(x)=\mbox{root}(A). $$
Write the inverse of $p(x)$ as a sequence of edges $(e_1,\dots, e_\ell)$, and let $\lambda(x)=(c(e_1), \dots, c(e_\ell))$. Now there is a unique $x_A\in A$ with the least $\lambda(x_A)$ in the lexicographic order. The assignment from $x\in A$ to this $x_A$ is a Borel selector. This shows that $R_D$ is smooth.

It remains to prove the last statement of Theorem~\ref{thm:markers}.  That is, if $f_1^D\cdots f_n^D(x)=f_1^D\cdots f_n^D(y)$, then $xR_Dy$.
\begin{claim}\label{laststate}
     For any $x,y\in F(X)$ and $1\le i\le n$, if $$f_1^{\frac14D_1-i(2D+1)}\cdots f_n^{\frac14D_1-i(2D+1)}(x)=f_1^{\frac14D_1-i(2D+1)}\cdots f_n^{\frac14D_1-i(2D+1)}(y),$$ then $xR_D^iy$.
\end{claim}
\noindent\emph{Proof of the claim}.
The conclusion holds for $i=0$ since if $f_1^{\frac14D_1}\cdots f_n^{\frac14D_1}(x)=f_1^{\frac14D_1}\cdots f_n^{\frac14D_1}(y)$,
then $xR_D^0y$.  Next, we prove the claim by induction.

Suppose the conclusion holds for $i-1$. Therefore, if $$f_1^{\frac14D_1-(i-1)(2D+1)-1}\cdots f_n^{\frac14D_1-(i-1)(2D+1)-1}(u)=f_1^{\frac14D_1-(i-1)(2D+1)-1}\cdots f_n^{\frac14D_1-(i-1)(2D+1)-1}(v),$$ then $u\in F_i^-([u]_{R_D^{i-1}})\iff v\in F_i^-([u]_{R_D^{i-1}})$.

Assume $$f_1^{\frac14D_1-i(2D+1)}\cdots f_n^{\frac14D_1-i(2D+1)}(x)=f_1^{\frac14D_1-i(2D+1)}\cdots f_n^{\frac14D_1-i(2D+1)}(y),$$ then we prove $xR_D^iy$. 

Note that $x\in F_i^-([x]_{R_D^{i-1}})\iff y\in F_i^-([x]_{R_D^{i-1}})$.

Consider the first case that $x,y\in F_1^-([x]_{R_D^{i-1}})$ or $x,y\in [x]_{R_D^{i-1}}\setminus \bigcup_{z\in F_i^-([x]_{R_D^{i-1}})}[z]_{R_D^i}$. For every $\beta\in\mathbb N$, we have 
$f_1^{\frac14D_1-i(2D+1)}\cdots f_n^{\frac14D_1-i(2D+1)}\bigl(f_i^\beta(x)\bigr)
=
f_1^{\frac14D_1-i(2D+1)}\cdots f_n^{\frac14D_1-i(2D+1)}\bigl(f_i^\beta(y)\bigr)$. Thus, $f_i^\beta(x)R_D^{i-1}f_i^\beta(y)$ according to the hypothesis and consequently, $L_x^i=L_y^i$.
Moreover, the assumption guarantees that $\rho_{f_i}(x,y)=0$, implying
$\mathrm{Lev}_i(x)=\mathrm{Lev}_i(y)$. By the definition of $R_D^i$, we have $xR_D^i y$.

We can state that the only remaining case is $x,y\in \bigcup_{z\in F_i^-([x]_{R_D^{i-1}})}[z]_{R_D^i}$. Since if $x\in \bigcup_{z\in F_i^-([x]_{R_D^{i-1}})}[z]_{R_D^i}$, then there are $z_1\in F_i^-([x]_{R_D^{i-1}})$ and $\beta\in (0,2D+1]\cap\mathbb{N}$ such that $f_i^\beta(z_1)=x$. For $z_2\in f_i^{-\beta}(y)$, we have $$f_1^{\frac14D_1-i(2D+1)}\cdots f_n^{\frac14D_1-i(2D+1)}\bigl(f_i^\beta (z_1)\bigr)
=
f_1^{\frac14D_1-i(2D+1)}\cdots f_n^{\frac14D_1-i(2D+1)}\bigl(f_i^\beta (z_2)\bigr).$$ 
Thus $z_2\in  F_i^-([x]_{R_D^{i-1}})$ by our statement, and $y\in[z_2]_{R_D^i}$ by the definition. Similar to the preceding argument, we also have $\mathrm{Lev}_i(z_1)=\mathrm{Lev}_i(z_2)$. It follows that $y\in [z_2]_{R_D^i}=[z_1]_{R_D^i}$. So $x R_D^iy$.
\hfill$\mbox{\qed}_{\mbox{\scriptsize Claim}}$

\medskip
Finally, since
$\frac14D_1-n(2D+1)>D$, the equality
$f_1^D\cdots f_n^D(x)=f_1^D\cdots f_n^D(y)$ implies
$f_1^{\frac14D_1-n(2D+1)}\cdots
f_n^{\frac14D_1-n(2D+1)}(x)
=
f_1^{\frac14D_1-n(2D+1)}\cdots
f_n^{\frac14D_1-n(2D+1)}(y)$.
Claim~\ref{laststate} therefore gives $xR_D^ny$.
Since $R_D=R_D^n$, it follows that $xR_Dy$.
\end{proof}

\begin{remark}
\label{rem:uniform-upper-bound}
Although it is not stated in the statement of Theorem~\ref{thm:markers}, we note that the $R_D$-classes constructed in the proof of Theorem~\ref{thm:markers} have uniformly bounded $f_i$-diameters for all $1\leq i\leq n$. In fact, once the parameter $D$ is given,  all auxiliary parameters, in particular the large scale $D_1$, are fixed. The construction then obtains face-separated regions of uniformly bounded size of about $2CD_1$ and further operations are all cuts which respect this bound. Thus there exists $M_D<\infty$ such that
$$
d_i(A)\le M_D
$$
for every $R_D$-class $A$ and every $1\le i\le n$. Indeed, we can take $M_D=2CD_1$.
\end{remark}

%---------------------
\subsection{Rooted marker regions}

We next refine the rootless regions in Theorem~\ref{thm:markers} further
to obtain rooted regions whose side lengths in all $n$ directions equal to $d$ or $d+1$.

\begin{theorem}\label{thm:rooted-markers}
    Let $X$ be a standard Borel space, and let $d\in \mathbb{N}^+$. For $n\in \mathbb{N}$, suppose that $f_1,\cdots ,f_n:X\to X$ are countable-to-one commuting Borel functions. Assume that for every $r\in\mathbb N^+$, there is a Borel $r$-forward-independent
hitting set with syndeticity $Cr$ for
$G_{f_1,\ldots,f_n}\upharpoonright F(X)$,
where $C\in\mathbb N^+$ is independent of $r$. Then there exists a smooth Borel subequivalence relation
$R_d\subseteq E_{f_1,\cdots,f_n}\upharpoonright F(X)$
such that every $R_d$-class $A$ is a rooted region satisfying
 $d_i(A)=d$ or $d_i(A)=d+1$ for $1\le i \le n$.
\end{theorem}

\begin{proof}
Choose an integer $D>d^2$, and let
$R'_D\subseteq E_{f_1,\cdots,f_n}\upharpoonright F(X)$
be the Borel subequivalence relation given by Theorem~\ref{thm:markers}.
Thus every $R'_D$-class is a rootless region $A$ satisfying
 $d_i(A)\ge D$ for $1\le i \le n$.

We again define a sequence of equivalence relations $R^0_d, R^1_d, \dots, R^n_d$. Let $R_d^0=R'_D$. Each $R^i_d$ is a refinement of $R^{i-1}_d$. For each $0\leq i\leq n$, we maintain the inductive hypothesis that for $j< i$, the $f_{j}$-side lengths of the $R^{j}_d$-classes equal to $d$ or $d+1$ and for $j\ge i$, the $f_{j}$-side lengths of the $R^i_d$-classes are at least $D$. 

Now suppose $x\in F(X)$ and let $B=[x]_{R^{i-1}_d}$. We define the $R^i_d$-classes which are subsets of $B$. By the inductive hypothesis, $d_i(B)\geq D$. By an elementary fact in number theory, we can write
$$
d_i(B)=ad+b(d+1)\quad \mbox{ for some integers $a,b\ge 0$.}
$$
At least one of $a$ and $b$ is positive; without loss of generality, we assume $a>0$. For $y, z\in B$, set
$$yR_d^{i}z\iff \left\lfloor \frac{L_y^i}{d}\right\rfloor=\left\lfloor\frac{L_z^i}{d}\right\rfloor\leq a 
\mbox{ or } \left\lfloor\frac{L_y^i-ad}{d+1}\right\rfloor=\left\lfloor\frac{L_z^i-ad}{d+1}\right\rfloor\ge0.$$
This finishes the definition of $R_d^i$. It is clear that the inductive hypothesis is maintained.

We illustrate the construction in Figure~\ref{Fig.7}. When $n=2$, $R_d^1$ is obtained by cutting each rootless
region $[x]_{R_d^0}$ into strips in the $f_1$-direction. The resulting
$f_1$-edge lengths are $d$ or $d+1$. Then $R_d^2$ is obtained by
cutting each $R_d^1$-class into strips in the $f_2$-direction, again with
edge lengths $d$ or $d+1$; see Figure~\ref{Fig.7}. The figure is drawn in the case where $f_1$ is injective and $f_2$ is
countable-to-one.

\begin{figure}[H]
\centering
 \includegraphics[trim=3.5cm 18.1cm 6cm  2.5cm,clip,width=0.6\linewidth]{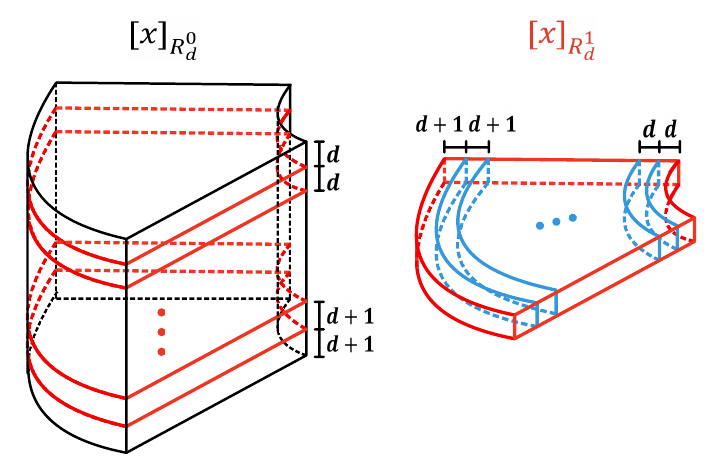}
\caption{A two-step refinement of $[x]_{R_d^0}$.
}
\label{Fig.7}
\end{figure}

Each $R_d^{n}$-class $A$ is again a rootless region, and by construction it now satisfies
$$
d_i(A)=d\text{ or } d_i(A)=d+1
\quad\text{for}\quad
1 \le i \le n.
$$
We finally split each $R_d^n$-class into rooted regions. Let $A$ be an
$R_d^n$-class. For each
$$y\in \bigcap\limits_{1\le i\le n} F_i^{+}(A),$$
define
$$[y]_{R_d}
\ :=\
\bigl\{z\in A:\ \exists\, \alpha_1,\cdots, \alpha_n\in\mathbb{N}\ \text{ such that }\ f_1^{\alpha_1}\cdots f_n^{\alpha_n}(z)=y\bigr\}.$$
Then $R_d$ is a Borel subequivalence relation of $R_d^{n}$, and every $R_d$-class $B$ is a rooted region satisfying
$d_i(B)=d$ or $d_i(B)=d+1$
for $1 \le i \le n$.

The assignment from each $R_d$-class $B$, which is a rooted region, to its root, $\mbox{root}(B)$, is obviously a Borel selector. Thus $R_d$ is smooth.
\end{proof}

%-------------------------
\section{Applications}
\subsection{Borel asymptotic dimension}
Given a locally countable Borel graph, one of the main problems is to determine whether its 
connectedness relation is hyperfinite. In this subsection, we prove
a finite Borel asymptotic dimension result for the graphs considered in this
paper, and deduce their hyperfiniteness as a consequence.

The following theorem was proved in \cite{GrebikHiggins2026}.

\begin{theorem}[{\cite[Theorem 1.2]{GrebikHiggins2026}}]
\label{thm:singlefunction}
Let $X$ be a standard Borel space, and let $f:X\to X$ be a
countable-to-one Borel function. Then $G_f\upharpoonright F(X)$ has finite
Borel asymptotic dimension if and only if for every $r\in\mathbb N^+$, there
exists a Borel $r$-forward-independent hitting set.
\end{theorem}

%The above theorem characterizes finite Borel asymptotic dimension for graphs generated by a single countable-to-one Borel function on the free part. 
Here we show a strengthened version in which the hitting set is required to have a linear syndeticity bound.

\begin{proposition} \label{prop:42}
     Let $X$ be a standard Borel space, and let $f:X\to X$ be a
countable-to-one Borel function. Then $G_f\upharpoonright F(X)$ has finite Borel asymptotic dimension if and only if, for every $r\in\mathbb N^+$, there is a Borel $r$-forward-independent hitting set with syndeticity $4r$.
\end{proposition}

\begin{proof}
 By Theorem~\ref{thm:singlefunction}, it suffices to show that for any $r\in \mathbb N^+$, if there is a Borel $r$-forward-independent hitting set, then there is a Borel $r$-forward-independent hitting set with syndeticity $4r$. 
 
Fix $r\in \mathbb N^+$. Let $H_0$ be a Borel $r$-forward-independent hitting set. Inductively define 
$$H_{m+1}=H_m\cup\left\{f^{k-2r}(x)\colon x\in F(X), \mbox{$k\in\mathbb{N}$ is the least such that $f^k(x)\in H_m$, and $k\geq 4r$}\right\}.$$
 Thus we have an increasing sequence of Borel subsets of $F(X)$. Let $H=\bigcup_m H_m$. 

We verify that $H$ is a Borel $r$-independent hitting set with syndeticity $4r$. 

 Note that if $x\in H_m\setminus H_{m-1}$, then  $f^{2r}(x)\in H_{m-1}$ and there is no $j<2r$ such that $f^j(x)\in H_{m-1}$. Thus, for any $x\in H$, there is no $j<2r$ such that $f^j(x)\in H$. Otherwise for $x\in H_m\setminus H_{m-1}$, if $f^j(x)\in H$ for some $0<j<2r$, we have that $f^j(x)\in H_k\setminus H_{k-1}$ for some $k\ge m$. However, $f^{2r}(x)\in H_{m-1}$ leads to a contradiction since $f^j(x)\in H_k$, and $f^{2r-j}(f^j(x))\in H_{m-1}\subseteq H_{k-1}$.
It follows that $H$ is $r$-forward-independent.

Furthermore, for any $x\notin H$, by the hitting property of $H_0$, there is some $k\in\mathbb{N}$ which is the least such that $f^k(x)\in H_0$. By our construction, for
$$m=\left\lfloor \frac{k}{2r}\right\rfloor-1,$$ 
there is some $2r\leq j<4r$ such that $f^j(x)\in H_m$. Thus $H$ has syndeticity $4r$.
\end{proof}

%Thus, in the single-function case, finite Borel asymptotic dimension is equivalent to the existence of $r$-forward-independent hitting sets with a linear syndeticity bound. 
%The next theorem shows that the similar marker condition is sufficient for finite Borel asymptotic dimension for graphs generated by finitely many commuting Borel functions.

The next theorem generalizes the sufficiency direction of the above proposition to graphs generated by finitely many commuting Borel functions.

\begin{theorem}
\label{thm:asdim}
Let $X$ be a standard Borel space. For $n\in\mathbb{N}$, suppose that $f_1,\cdots,f_n:X\to X$ are countable-to-one commuting Borel functions, and let $\rho$ be the graph metric on $G_{f_1,\cdots,f_n}\upharpoonright F(X)$. Assume that for every $r\in\mathbb N^+$, there is a Borel $r$-forward-independent hitting set with syndeticity $Cr$ for $G_{f_1,\ldots,f_n}\upharpoonright F(X)$, where $C\in\mathbb N^+$ is independent of $r$. 
Then $(F(X),\rho)$ has finite strong Borel asymptotic dimension. In particular,
$E_{f_1,\cdots,f_n}\upharpoonright F(X)$ is hyperfinite.
\end{theorem}
\begin{proof}
Fix an integer $r$ and $ D> 2r$. Let $E_r:=R_D\subseteq E_{f_1,\cdots,f_n}\upharpoonright F(X)$
be the smooth Borel subequivalence relation given by Theorem~\ref{thm:markers} for $D$. Then every $E_r$-class is a rootless region $A$ satisfying
$d_i(A) \ge D$
for $1 \le i \le n$.
Moreover, we have $ xE_r y$ if $f_1^D\cdots f_n^D(x)=f_1^D\cdots f_n^D(y)$. 
By Remark~\ref{rem:uniform-upper-bound}, the $E_r$-classes have uniformly
bounded side lengths. Since each $E_r$-class is a rootless region, this implies
that the $E_r$-classes have uniformly bounded $\rho$-diameter.

Next we show that every ball $B_{\rho}(x,r)$ meets at most $2^n$ many $E_r$-classes. Fix $x\in F(X)$. By the property of $E_r$, we have that $f_1^{-2r}\cdots f_n^{-2r}(f_1^r\cdots f_n^r(x))$ is contained in a single $E_r$-class. Assume first that
$f_1^{-2r}\cdots f_n^{-2r}
\bigl(\{f_1^r\cdots f_n^r(x)\}\bigr)
\neq\varnothing$,
and choose $y_0$ in this set. We claim that
$$ B_{\rho}(x, r)\subseteq \bigcup_{j_1, \dots, j_n\in\{0,2\}} \left[f_1^{j_1r}\cdots f_n^{j_nr}(y_0)\right]_{E_r}, $$
which clearly implies that $B_{\rho}(x, r)$ meets at most $2^n$ many $E_r$-classes. To see the claim, let $z\in B_{\rho}(x, r)$. Then there are $\alpha_1,\dots, \alpha_n, \beta_1, \dots, \beta_n\in [0, r]$ such that $f_1^{\alpha_1}\dots f_n^{\alpha_n}(z)=f_1^{\beta_1}\cdots f_n^{\beta_n}(x)$. Since $f_1^{2r}\cdots f_n^{2r}(y_0)=f_1^r\cdots f_n^r(x)$, we have that 
$$ f_1^{\alpha_1+r-\beta_1}\cdots f_n^{\alpha_n+r-\beta_n}(z)=f_1^r\cdots f_n^r(x)=f_1^{2r}\cdots f_n^{2r}(y_0). $$
Since $z, y_0\in F(X)$, this implies that there are $h_1, \dots, h_n\in [0, 2r]$ such that $zE_rf_1^{h_1}\cdots f_n^{h_n}(y_0)$. Since the right side set $\bigcup_{j_1, \dots, j_n\in\{0,2\}} \left[f_1^{j_1r}\cdots f_n^{j_nr}(y_0)\right]_{E_r}$ is $E_r$-invariant, it suffices
to prove that
$f_1^{h_1}\cdots f_n^{h_n}(y_0)$
belongs to it.
Without loss of generality, we can assume $z=f_1^{h_1}\cdots f_n^{h_n}(y_0)$. 
Let $y_1=f_2^{h_2}\cdots f_n^{h_n}(y_0)$. Then $z=f_1^{h_1}(y_1)$. Since each $E_r$-class is a rootless region with $f_1$-side length at least $D>2r$, and since $0\leq h_1\leq 2r$, we have that the set $\{y_1, f_1(y_1), \dots, f_1^{h_1}(y_1)=z\}$ is contained in at most two $E_r$-classes. If it is contained in a single $E_r$-class, we actually have $zE_ry_1$, or $z\in [y_1]_{E_r}$; if it is contained in two different $E_r$-classes, then we have $zE_r f_1^{2r}(y_1)$, or $z\in [f_1^{2r}(y_1)]_{E_r}$. Thus for some $j_1\in\{0, 2\}$, $z\in [f_1^{j_1r}(y_1)]_{E_r}$. Now let $y_2=f_1^{j_1r}f_3^{h_3}\cdots f_n^{h_n}(y_0)$. Then $f_1^{j_1r}(y_1)=f_2^{h_2}(y_2)$. By a similar argument, we get $j_2\in\{0,2\}$ such that $f_1^{j_1r}(y_1)\in [f_2^{j_2r}(y_2)]_{E_r}$. This implies that $z\in [f_2^{j_2r}(y_2)]_{E_r}$. Repeating this argument, we obtain $j_1, \dots, j_n\in\{0,2\}$ such that $z\in [f_1^{j_1r}f_2^{j_2r}\cdots f_n^{j_nr}(y_0)]_{E_r}$. The claim is proved. 

If $f_1^{-2r}\cdots f_n^{-2r}
\bigl(f_1^r\cdots f_n^r(x)\bigr)
=\varnothing$, then some of the points used in the preceding construction do not
exist. The corresponding $E_r$-classes are therefore absent, so
the number of $E_r$-classes meeting $B_\rho(x,r)$ can only
decrease. Hence it is still at most $2^n$.

Thus, for every $r>0$, we have constructed a smooth Borel equivalence relation $E_r$ whose classes have uniformly bounded diameter and such that every ball of radius $r$ meets at most $2^n$ $E_r$-classes. Hence
$$
\operatorname{asdim}^*_B(F(X),\rho)\le 2^n-1<\infty.
$$
The hyperfiniteness of $E_{f_1,\cdots,f_n}\upharpoonright F(X)$ now follows from Lemma~\ref{lem:hyp}.
\end{proof}

\begin{corollary}[{\cite[Theorem 1.6]{NaryshkinShinkoWeilacherYuCommutingFunctions}}] Let $X$ be a standard Borel space,  let $f_1, \dots, f_n\colon X\to X$ be commuting Borel functions each of which is bounded-to-one, and let $\rho$ be the graph metric on $G_{f_1,\dots, f_n}\upharpoonright F(X)$. Then $(F(X), \rho)$ has finite Borel asymptotic dimension. In particular, $E_{f_1,\dots, f_n}\upharpoonright F(X)$ is hyperfinite.
\end{corollary}

\begin{proof} Since $G_{f_1,\dots, f_n}\upharpoonright F(X)$ is locally finite, the notions $\operatorname{asdim}_B(F(X), \rho)$ and $\operatorname{asdim}_B^*(F(X), \rho)$ are the same. It follows from Theorems~\ref{thm:singlefunction} and \ref{thm:Weilacher} that $(F(X), \rho)$ has finite Borel aymptotic dimension, which implies that  $E_{f_1,\dots, f_n}\upharpoonright F(X)$ is hyperfinite, either by \cite[Theorem 7.1]{ConleyJacksonMarksSewardTuckerDrob2023} or by Lemma~\ref{lem:hyp}.
\end{proof}

We also have the following corollary strengthening Theorem~\ref{thm:singlefunction} and Proposition~\ref{prop:42}.

\begin{corollary} Let $X$ be a standard Borel space, and let $f:X\to X$ be a
countable-to-one Borel function. Then the following are equivalent:
\begin{enumerate}[label=(\roman*),leftmargin=2.7em]
\item $G_f\upharpoonright F(X)$ has finite Borel asymptotic dimension.
\item $G_f\upharpoonright F(X)$ has finite strong Borel asymptotic dimension. 
\item For every $r\in\mathbb N^+$, there exists a Borel $r$-forward-independent hitting set. 
\item For every $r\in\mathbb N^+$, there is a Borel $r$-forward-independent hitting set with syndeticity $4r$.
\end{enumerate}
\end{corollary}

\begin{proof} By definition, (ii)$\Rightarrow$(i) holds. Theorem~\ref{thm:asdim} shows (iv)$\Rightarrow$(ii). 
\end{proof}

\begin{corollary}
\label{coro:color}
   Let $X$ be a standard Borel space. For $n\in\mathbb{N}$, suppose that $f_1,\cdots,f_n:X\to X$ are finite-to-one commuting Borel functions. Assume that for any $r\in\mathbb N^+$, there is a Borel $r$-forward-independent hitting set with syndeticity $Cr$ for $G_{f_1,\ldots,f_n}\upharpoonright F(X)$, where $C\in\mathbb N^+$ is independent of $r$. Then $\chi_B(G_{f_1,\cdots,f_n}\upharpoonright F(X))\leq 3$.
\end{corollary}

\begin{proof}
Since each $f_i$ is finite-to-one and there are only finitely many generators,
the graph
$G_{f_1,\ldots,f_n}\upharpoonright F(X)$
is locally finite.

We first observe that this graph has no odd cycles. Indeed, suppose that
$x_0,x_1,\ldots,x_m=x_0$
is a cycle. Let $\rho$ be the graph metric on $G_{f_1,\cdots,f_n}\upharpoonright F(X)$, we can state that $|\rho(x_0,x_k)-\rho(x_0,x_{k+1})|=1$ for any $0\le k\le m$. Thus, the number of $x_j$ such that $\rho(x_0,x_{j-1})<\rho(x_0,x_j)$ is equal to the number of $x_l$ such that $\rho(x_0,x_{l-1})>\rho(x_0,x_l)$, implying $m$ is even.
Therefore, the graph has no odd cycles, and so
$\chi\bigl(G_{f_1,\ldots,f_n}\upharpoonright F(X)\bigr)\le 2$.

By Theorem~\ref{thm:asdim},
$\operatorname{asdim}^*_B(F(X),\rho)<\infty$, which implies that $\operatorname{asdim}_B(F(X), \rho)<\infty$.
Since the graph is locally finite, Lemma~\ref{lem:color} gives
$$
\chi_B\bigl(G_{f_1,\ldots,f_n}\upharpoonright F(X)\bigr)
\le
2\chi\bigl(G_{f_1,\ldots,f_n}\upharpoonright F(X)\bigr)-1
\le 3.
$$
\end{proof}

%---------------------
\subsection{Borel perfect matching}
%In~\cite{GaoJacksonKrohneSeward2024}, Gao, Jackson, Krohne, and Seward proved the existence of Borel perfect matchings for Schreier graphs arising from free Borel actions of $\mathbb Z^n$. Their proof uses orthogonal marker regions. In the present setting, the generators need not be invertible, so such orthogonal structures are not directly available. We instead use the marker decompositions developed above and perform local modifications near the boundaries of the marker regions.
In this subsection, we first consider the case of two commuting Borel functions, where one is injective and the other one is bounded-to-one. In the following, we prove again the existence of $r$-forward-independent hitting set with a slightly improved syndeticity.

\begin{proposition}
\label{prop:specialcase}
Let $X$ be a standard Borel space, and let
$f_1,f_2:X\to X$ be commuting Borel functions such that $f_1$ is injective
and $f_2$ is bounded-to-one. Then, for every $r\in\mathbb N^+$, there is a
Borel $r$-forward-independent hitting set with syndeticity $6r$ for
$G_{f_1,f_2}\upharpoonright F(X)$.
\end{proposition}
\begin{proof}
Fix $r\in\mathbb N^+$. 
 Since $f_1$ is injective and $f_2$ is
bounded-to-one, define a directed graph $\vec G=(F(X),\vec R)$ by
$$
x\,\vec R\,y
\quad\Longleftrightarrow\quad
x\neq y
\ \text{ and }\
\exists a\in[-r,r]\cap \mathbb{Z},\ \exists b\in[0,r]\cap \mathbb{Z}\ 
\bigl(f_1^a f_2^b(x)=y\bigr).
$$
Then $\vec G$ has bounded degree and hence finite Borel chromatic
number. Applying Lemma~\ref{lem:qk},
we obtain a Borel quasi-kernel
$H\subseteq F(X)$
for $\vec G$. We claim that $H$ is the desired set.

First, $H$ is $r$-forward-independent. Indeed, suppose that distinct
$x,y\in H$ satisfy
$f_1^{\alpha_1}f_2^{\alpha_2}(x)
=
f_1^{\beta_1} f_2^{\beta_2}(y)$,
where
$\alpha_i, \beta_i \in [0,r]\cap \mathbb{N}$ and
$\alpha_i\beta_i=0$ for $i\in \{1,2\}$.
Since $f_1$ is injective, $f_1$-powers can be cancelled, and we obtain either
$f_1^a f_2^b(x)=y$ or
$f_1^a f_2^b(y)=x$
for some $a\in[-r,r]\cap \mathbb{Z}$ and $b\in[0,r]\cap \mathbb{Z}$. Hence $x\,\vec R\,y$ or
$y\,\vec R\,x$, contradicting the $\vec R$-independence of $H$.
Therefore, $H$ is $r$-forward-independent.

It remains to prove that $H$ has syndeticity $6r$. Fix
$x\in F(X)$ and consider $f_1^{2r}(x)$. If
$f_1^{2r}(x)\in H$, then the desired conclusion follows
immediately, since $H$ is reached from $x$ in $2r$ forward steps. If $f_1^{2r}(x)\notin H$, then there are $y,z\in F(X)$ with $f_1^{2r}(x)\vec R y\vec R z$, such that we have either (a) $y\in H$, or (b) 
$z\in H$,
since $H$ is a quasi-kernel of $\vec G$, 

If (a) holds, i.e., $y\in H$. We can write $y=f_1^{2r+j_1}f_2^{j_2}(x)$ for some $j_1\in [-r,r]\cap \mathbb{Z}$ and $j_2\in [0,r]\cap \mathbb{Z}$, then from $x$ we reach $y\in H$ in at most
$2r+j_1+j_2\le 4r\le 6r$
forward steps. Suppose (b) holds, i.e., $x\vec R\,y\vec R\,z$ for some $z\in H$. Then for some
$k_1\in[-r,r]\cap \mathbb{Z}$ and $k_2\in[0,r]\cap \mathbb{Z}$,
$z=f_1^{k_1} f_2^{k_2}(y)$.
Hence
$z=f_1^{2r+j_1+k_1} f_2^{j_2+k_2}(x)$, and the total
number of forward steps is at most
$(2r+j_1+k_1)+(j_2+k_2)
\le (2r+r+r)+(r+r)
=6r$. 
\end{proof}

We give a simple example of two commuting Borel functions where one of them is an injection and the other one is exactly 2-to-1.

\begin{example}
Let $X=2^\mathbb{Z}\times2^\mathbb{N}$. For $x\in X$, define
$$
f_1(x,y)=(\sigma_{\mathbb{Z}}(x),y),
\qquad
f_2(x,y)=(x,\sigma_{\mathbb{N}}(y)),
$$
where $\sigma_\mathbb{Z}$ is on $2^\mathbb{Z}$ with  $\sigma_{\mathbb{Z}}(x)(n)=x(n+1)$ and $\sigma_\mathbb{N}$ is on $2^\mathbb{N}$ with $\sigma_{\mathbb{N}}(y)(n)=y(n+1)$.
Then $f_1$ and $f_2$ are Borel functions, $f_1$ is bijective, $f_2$ is exactly 2-to-1, and
$f_1\circ f_2=f_2\circ f_1$.
\end{example}

We introduce some additional notation to be used below. For $1\le i\le n$ and distinct points $x,y\in X$, we say that
$y$ is an \emph{$f_i$-sibling} of $x$ if
$f_i(y)=f_i(x)$. When $f_i$ is exactly $2$-to-$1$, every point
$x\in X$ has a unique $f_i$-sibling. 
% The following theorem shows the existence of perfect matchings for $G_{f_1,f_2}$ under an additional assumption. 

By \cite[Proposition~2.8]
{ConleyMiller2017MeasurablePerfectMatchings}, if $f$ is a countable-to-one Borel
surjection, then $G_f$ has a Borel perfect matching outside the
injective part of $f$. In particular,
$G_f$ has a Borel perfect matching when $f$ is an even-to-one
Borel surjection. Without surjectivity, this conclusion is false
in general. The next theorem shows that surjectivity is no longer
needed after adding a commuting injective map and restricting to
the free part.

\begin{theorem}
\label{thm:perfect-matching}
Let $X$ be a standard Borel space. Suppose that $f_1,f_2:X\to X$ are commuting Borel functions such that $f_1$ is injective and $f_2$ is bounded-to-one and exactly even-to-one.
Then there exists a Borel perfect matching for $G_{f_1,f_2}\upharpoonright F(X)$.
\end{theorem}

\begin{proof}
%We note that the non-surjective case can be reduced to the surjective case. Our construction will not depend on whether  $f_2^{-1}(x)\neq \varnothing$, and so we may assume $f_2^{-1}(x)\neq \varnothing$ and the proof will not change. If $f_1^{-1}(x)=\varnothing$ then we can simply match points in $\{f_2^j(x):j\ge 0\}$ in a trivial way since we have a point to start, and the points remain to be matched is like the first case. Thus, we assume that $f_1,f_2$ are surjective.

We first assume that $f_1, f_2$ are surjective and that $f_2$ is exactly $2$-to-$1$. By Theorem~\ref{thm:Weilacher} or Proposition~\ref{prop:specialcase}, this case satisfies the condition in Theorem~\ref{thm:markers}.
Fix an integer $d>0$, and choose $D\gg d$.
Let
$R'_D\subseteq E_{f_1,f_2}\upharpoonright F(X)$
be the Borel subequivalence relation given by Theorem~\ref{thm:markers} when $n=2$. Recall that every $R'_D$-class $A$ is a rootless region with $d_1(A),d_2(A)\ge D$. Then, as in the proof of Theorem~\ref{thm:rooted-markers}, we can cut the $R'_D$-classes into regions with $f_2$-side lengths $d$ or $d+1$, and denote the resulting Borel subequivalence relation by $R'_d$. 

We further define $R_d\subseteq R_d'$ as follows. For $x\in F_1^+([x]_{R_d'})\cap F_2^+([x]_{R_d'})$, let 
$$[x]_{R_d}=\left\{y\in [x]_{R_d'}:\exists j_1,j_2\in\mathbb{N}\ \mbox{ such that } f_1^{j_1}f_2^{j_2}(y)=f_2(x)\right\}.$$
Thus, each $R_d$-class is a rootless region with $\rho_{f_2}(\mathrm{root}([x]_{R_d}),[x]_{R_d})=1$.

For each $x\in F(X)$, let $x'$ denote the \emph{sibling} of $x$. 
We have
$x\;R_d\;x'\text{ for every }x\in F(X)$. 

Fix a Borel linear ordering $\prec$ on $F(X)$. We now define a Borel partial matching $M$ on $F(X)$ as follows.
Whenever
$x\prec x'$ and
$x'\in F_1^-([x]_{R_d})$,
we put
$$
\left\{f_1^{2k}(x),f_1^{2k+1}(x)\right\}\in M
\quad\text{and}\quad
\left\{f_1^{2k+1}(x'),f_1^{2k+2}(x')\right\}\in M
$$
for every $k$ such that $2\le 2k+2\le d_1([x]_{R_d})$.
By construction, $M$ is a Borel partial matching satisfying the following properties:
\begin{enumerate}[label=\roman*),leftmargin=2.7em]
\item If $x\notin\dom(M)$, then $x\in F_1([x]_{R_d})$.
\item If $x,f_1(x)\in\dom(M)$, then
$
\{x,f_1(x)\}\in M
\iff
\{x',f_1(x')\}\notin M.$

\item There is no $f_2$-edge in $M$.
\end{enumerate}

For any partial matching $N$ of $F(X)$ and $x\in F(X)$, define
$$
p_N(x)=
\begin{cases}
(1,1), & \text{if } x,f_1(x)\in \dom(N),\\
(1,0), & \text{if } x\in\dom(N),\ f_1(x)\notin\dom(N),\\
(0,1), & \text{if } x\notin\dom(N),\ f_1(x)\in\dom(N),\\
(0,0), & \text{if } x,f_1(x)\notin\dom(N).
\end{cases}
$$
Thus, $p_N(x)$ keeps track of the matched/unmatched statuses of $x$ and $f_1(x)$ with respect to $N$. 

Next we specify a procedure to modify $M$ into a perfect matching $N$. The procedure is performed near points $x$ with $p_M(x)=(1,0)$ or $(0,1)$. Thus, if $p_M(x)=(0,0)$ or $(1,1)$, then do nothing. 

Consider the case $p_M(x)=(1,0)$, in other words, $x$ is matched in $M$ and $f_1(x)$ is not. 
By the construction of $M$, this is equivalent to
$p_M(x')=(0,1)$. The left panel of Figure~\ref{Fig.8} illustrates the situation.
We are now going to modify the partial matching $M$ in a local area near $x$ and $x'$, so that in the resulting partial matching, which we denote as $M_x$, the matched/unmatched status of $f_2(x)$ is unchanged, but both $x$ and $x'$ are taken care of.  More precisely, we will have $$
p_{M_x}(x)=(0,0),\qquad
p_{M_x}(x')=(1,1),\qquad 
p_{M_x}(f_2(x))=p_M(f_2(x)).
$$
To specify the local area, choose $m\in [5,d_1([x]_{R_d})]\cap \mathbb{N}$,
and let $u\in [x]_{R_d}$ satisfy
$f_1^m(u)=x$.
Define
$$
\mathcal A_m(x)
   :=\{f_1^j(u),\,f_1^j(u')\colon 0\leq j\leq m\}\cup \{f_1^jf_2(u)\colon 0\le j\le 5\}.
$$
The exact modifications performed to $\{f_1^j(u), f_1^j(u'), f_1^jf_2(u)\colon 0\leq j\leq 5\}$ are illustrated in the right panel of Figure~\ref{Fig.8}. Then further modifications are extended along the $f_1$-segments connecting $u$ to $x$ and $u'$ to $x'$, respectively.

For the case $p_M(x)=(0,1)$, the modifications for $x'$ are performed to obtain $M_{x'}$ so that
$$
p_{M_{x'}}(x)=(1,1),\qquad
p_{M_{x'}}(x')=(0,0),\qquad
p_{M_{x'}}(f_2(x))=p_M(f_2(x)).
$$

\begin{figure}[H]
    \centering
    \resizebox{0.6\linewidth}{!}{\input{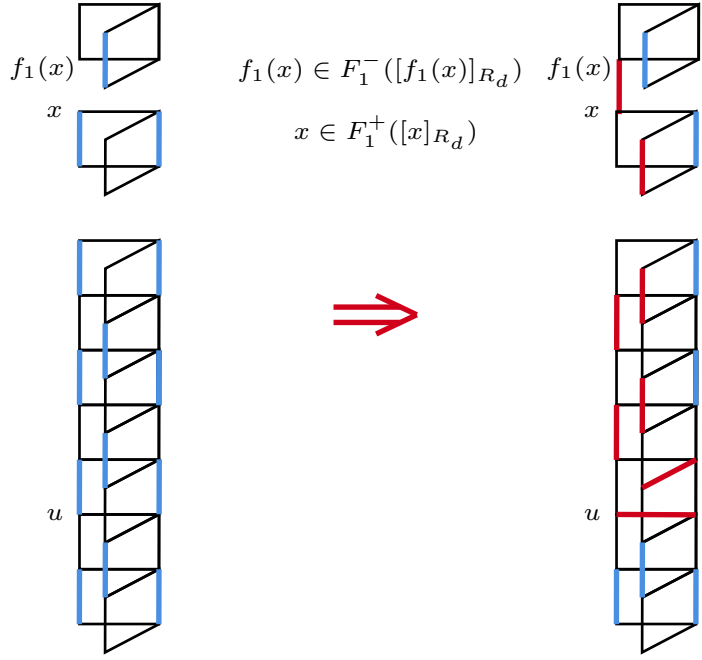}}
    \caption{A local modification inside the area $\mathcal A_m(x)$.}
    \label{Fig.8}
\end{figure}

Now, to modify the partial matching $M$ to obtain a new partial matching $M'$ by using the local modification above, we set
the $m$-values for the local areas significantly differently for all pairs of $x$ and $f_2(x)$ in the same $R_d$-classes. 
More precisely, for $x\in F_1^+([x]_{R_d})\setminus F_2^+([x]_{R_d})$, we ensure $\mathcal{A}_m(x)\cap \mathcal{A}_{m'}(f_2(x))=\varnothing$ by setting $m>m'+10$.  This is possible because we have for all $x\in F(X)$, $d_1([x]_{R_d})\geq D\gg d+1\geq d_2([x]_{R_d})$. And for $x\in F_2^+([x]_{R_d})$, we just need to make sure that there are only $f_1$-edges matched in $\mathcal{A}_m(x)\cup \mathcal{A}_{m'}(f_2(x))$.

We now have that for any $x\in F(X)$, $p_{M'}(x)=(0,0)$ or $(1,1)$, and $p_{M'}(x)=(0,0)$ only if $x\in F_2^+([x]_{R_d})$. In a final round of modification, match up $x$ with $f_1(x)$ if $p_{M'}(x)=(0,0)$. The result is a perfect matching $N$ of $F(X)$. This finishes the proof assuming that $f_1, f_2$ are surjective and that $f_2$ is exactly $2$-to-$1$.

The proof is similar in the more general case that $f_2$ is bounded-to-one and exactly even-to-one. In this case, every point $x\in F(X)$ has an odd number of siblings, and it is possible to assign $x\mapsto x'$ so that $x'$ is a sibling of $x$, and $(x')'=x$. With this assignment, the proof works exactly as above. 

Next we consider the general cases in which $f_1$ or $f_2$ is not surjective. First, note that the above proof does not depend on whether $f_2^{-1}(x)=\varnothing$, and so the proof is exactly the same if $f_2$ is not surjective. 
Now suppose $f_1$ is not surjective. Let $A$ be the set of all $x\in F(X)$ such that there is $y\in [x]_{E_{f_1, f_2}}$ with $f_1^{-1}(y)=\varnothing$. Then $f_1$ is surjective as a function on $F(X)\setminus A$. It suffices to define a Borel perfect matching on $A$. Without loss of generality, assume $A=F(X)$, i.e., for any $x\in F(X)$ there is $y\in [x]_{E_{f_1,f_2}}$ with $f^{-1}(y)=\varnothing$. We define a perfect matching $M$ on $F(X)$ so that $M$ consists only of $f_1$-edges. 

For each $x\in F(X)$ let $L_x=E_{f_1}$. If there is $y\in L_x$ such that $f_1^{-1}(y)=\varnothing$, $L_x$ would be a ray with a starting point; otherwise it is bi-infinite. Now if $L_x$ is a ray with a starting point $z$, we put $\{f_1^{2k}(z), f_1^{2k+1}(z)\}\in M$ for all $k\in\mathbb{N}$. If $L_x$ is bi-infinite and if we are able to select a single element $z\in L_x$ in a Borel way, then we can put $\{f_1^{2k}(z), f_1^{2k+1}(z)\}\in M$ for all $k\in\mathbb{Z}$ and $M$ will be a perfect matching. Thus it suffices to specify a procedure to select a single point in $L_x$ in a Borel way. For this, consider a bi-infinite $L_x$. Consider the graph $G_L$ with the vertex set $V_L=\{L_y\colon y\in [x]_{E_{f_1,f_2}}\}$ and the edge relation $R_L$ defined by
$$ L_yR_LL_z\iff \exists u\in L_y\ \exists v\in L_z\ (f_2(u)=v \mbox{ or } f_2(v)=u). $$
Then $G_L$ is locally finite. Let $\rho_L$ be the graph metric of $G_L$. By our assumption, there is a ray $L_y$ with the starting point $y$ so that $\rho_{L}(L_y, L_x)$ is the least and so that if there is another ray $L_{y'}$ with the starting point $y'$ and with $\rho_{L}(L_y, L_x)=\rho_{L}(L_{y'}, L_x)$, then $f_1^{\alpha_1}f_2^{\alpha_2}(y)=f_2^{\beta_2}(y')$ for some $\alpha_1, \alpha_2, \beta_2\in\mathbb{N}$. Intuitively, the $f_1$-level of $y$ is no higher than the $f_1$-level of $y'$. Although such $L_y$ might not be unique, the $f_1$-level of its starting point is uniquely determined. Let $z\in L_x$ be the point whose $f_1$-level is the same as $y$, in other words, $\rho_{f_1}(z, y)=0$. We have thus selected a single point in $L_x$ in a Borel way.
\end{proof}

\begin{corollary} Let $X$ be a standard Borel space and $I\subseteq \mathbb{N}$ is a countable index set.
Suppose that $\{f_i\}_{i\in I}$ is a family of Borel functions on $X$ such that $f_1\circ f_2=f_2\circ f_1$, $f_1$ is injective and $f_2$ is bounded-to-one and exactly even-to-one. Then there exists a Borel perfect matching for $G_{\{f_i\}_{i\in I}}\upharpoonright F(X)$.
\end{corollary}

%----------------
\subsection{Borel edge coloring}
%Borel edge colorings are closely related to Borel perfect matchings. For a $k$-regular graph, if its edge chromatic number is equal to the degree, then it has a Borel perfect matching. Thus, for an injective function $f_1 $ and  an exactly 2-to-1 function $f_2$, we want to have a Borel $5$-edge coloring for $G_{f_1,f_2}$. 

In this final subsection we consider Borel edge colorings. First, we consider again the case of two commuting functions where one of them is injective and the other one is exactly $2$-to-$1$. In this case the graph has maximum degree $5$,  hence its
edge chromatic number is at least $5$. We do not know whether
$G_{f_1,f_2}\upharpoonright F(X)$ has Borel edge chromatic
number $5$, but we prove
the following Borel $6$-edge-coloring result. 
The proof of the following theorem is inspired by \cite{GaoWangWang2025ContinuousEdge}.

\begin{theorem}
\label{thm:color}
Let $X$ be a standard Borel space. Suppose that $f_1,f_2:X\to X$ are commuting Borel functions such that $f_1$ is injective and $f_2$ is exactly 2-to-1.
Then $\chi'_B(G_{f_1,f_2}\upharpoonright F(X))\leq 6$.
\end{theorem}

\begin{proof}
By Theorem~\ref{thm:Weilacher} or Proposition~\ref{prop:specialcase}, the hypothesis of
Theorem~\ref{thm:markers} is satisfied. Fix $D>0$, and let
$R_D\subseteq E_{f_1,f_2}\upharpoonright F(X)$
be the subequivalence relation given by Theorem~\ref{thm:markers}. Let $\prec$ be a Borel linear ordering of $F(X)$. Let the set
of colors be
$\{1,2,3,c_1,c_2,c_3\}$.

We define the desired coloring $c$ by coloring the edges in several stages.
Let $x'$ denote the sibling of $x$. Then we always have $xR_Dx'$. If
$x,x'\in F_1^+([x]_{R_D})\cap F_2^+([x]_{R_D})$
and $x\prec x'$, define
$$
c(\{x,f_2(x)\})=c_1,
\qquad
c(\{x',f_2(x')\})=c_2.
$$
If
$y,y'\in F_2^+([x]_{R_D})$ and $f_1^j(y)=x,f_1^j(y')=x'$ for some $0\le j\le diam_{f_1}([x]_{R_D})$, define
$$
c(\{y,f_2(y)\})=c_1,\qquad
c(\{y,f_2(y')\})=c_2
$$
Moreover, for every
$z\in F_1^+([z]_{R_D})$,
define
$$
c(\{z,f_1(z)\})=c_3.
$$
Thus, all edges in between different $R_D$-classes have been colored. 

To color the edges within an $R_D$-class, we first color the $f_2$-edges whose endpoints lie on the
$F_1^+$-face of the $R_D$-class, using the colors $1,2,3$. This is possible since these $f_2$-edges form a binary forest, and therefore a greedy algorithm will do. The
remaining $f_2$-edges are colored columnwise: for an $f_2$-edge
$\{y,f_2(y)\}$, choose the unique point
$x\in F_1^+([y]_{R_D})$
such that $f_1^j(y)=x$ for some $j\ge 0$, and define
$$
c(\{y,f_2(y)\})=c(\{x,f_2(x)\}).
$$

It remains to color the remaining $f_1$-edges. Starting from an $x\in F_1^-([x]_{R_D})$ and moving along the $f_1$ forward direction, we use a greedy algorithm with the colors
$\{1,2,3,c_1,c_2\}$. Note that for every $f_1$-edge $\{y,f_1(y)\}$ to be colored, the colors already used on
edges incident with $y$ or $f_1(y)$ do not exhaust this set of five colors.
Thus at least one color is available. Since the remaining $f_1$-edges form
a finite chain inside the $R_D$-class, this greedy algorithm is Borel.

Therefore $c$ is a Borel proper edge coloring of
$G_{f_1,f_2}\upharpoonright F(X)$ with six colors. Hence
$\chi'_B\bigl(G_{f_1,f_2}\upharpoonright F(X)\bigr)\le 6$.
\end{proof}

The next result gives a general upper bound for the Borel edge chromatic number. This bound improves the general result from \cite{KechrisSoleckiTodorcevic1999} by $(n-1)(k-1)$.

\begin{theorem}
\label{thm:edgecoloring}
    Let $X$ be a standard Borel space. For $n\in\mathbb{N}$, suppose that $f_1,\cdots,f_n:X\to X$ are commuting Borel functions each of which is $k$-to-1. 
 Then $\chi'_B(G_{f_1,...,f_n}\upharpoonright F(X))\leq (n+1)(k+3)-5.$
\end{theorem}

% We can assume that $f_1,\cdots,f_n$ are surjective. 
The rest of this subsection is devoted to a proof of Theorem~\ref{thm:edgecoloring}. %Missing preimages only decrease the number of adjacent edges, so the coloring estimates below are unaffected. 
For an edge $e$ of $G_{f_1,\dots,f_n}$ and a region $A\subseteq F(X)$, we say that $e$ is \emph{adjacent to} $A$ if $|e\cap A|=1$.
We first prove a lemma that allows us to extend a partial proper edge coloring of the set of edges adjacent to a rooted region $A$ to a proper coloring of all edges in $A$.

\begin{lemma}\label{lem:localcolor}
     Let $X$ be a standard Borel space and $n,k\in\mathbb{N}^+$. Let $f_1,\cdots,f_n:X\to X$ be commuting Borel functions each of which is $k$-to-1. Let $A\subseteq F(X)$ be a rooted region with $d_i(A)>0$ for every $1\le i\le n$. Let $E=\{e\colon e\cap A\neq\varnothing\}$. 
Let $\Gamma_n$ be a set of colors with $|\Gamma_n|=(n+1)(k+3)-5$. Suppose $\Gamma_n=\left(\bigcup_{1\le i\le n}\Gamma_i^\partial\right)\sqcup \Gamma_n^{\mathrm{aux}},$ where the sets $\Gamma_i^\partial$ and $\Gamma_n^{\mathrm{aux}}$ are pairwise disjoint, with
$|\Gamma_i^\partial|=k
\text{ and }
|\Gamma_n^{\mathrm{aux}}|=3(n-1)+k+1$. Then, for any partial proper coloring $c_0$  whose domain is the set of edges adjacent to $A$ such that each $f_i$-edge is colored by $\Gamma_i^\partial$, there is a Borel proper edge coloring $c\colon E \rightarrow \Gamma_n$ satisfying
    \begin{enumerate}[label={\rm (\roman*)},leftmargin=3.2em]
     \item $c$ extends $c_0$ and is a proper edge coloring of $E$;
     \item if $e\in E$ and $c(e)\in \Gamma_i^\partial$ for some $1\le i \le n$, then either $e$ is an $f_i$-edge adjacent to $A$  or $e$ is an $f_{j}$-edge for $j > i$ and $e \subseteq A$. 
 \end{enumerate}
\end{lemma}

\begin{proof}
We prove it by induction on $n$.

The case $n=1$ is immediate. Assume now that $n>1$, and that the statement has already been proved for $n-1$ commuting  Borel functions each of which is $k$-to-$1$. For convenience, we may index the colors so that
$$
\Gamma_i^\partial=\{c_i^j:1\le j\le k\}
$$
for each $1\leq i\leq n$, and
$$
\Gamma_i^{\mathrm{aux}}
=
\{c_0^1,\dots,c_0^{k+1},c_1^{k+1},c_1^{k+2},c_1^{k+3},\dots,c_{i-1}^{k+1},c_{i-1}^{k+2},c_{i-1}^{k+3}\}
$$
for each $2\le i\le n$.

By the inductive hypothesis, there exists a proper edge coloring of all $f_i$-edges meeting $A$, for $1\le i\le n-1$, with values in
$$
\Gamma_{n-1}
=
\left(\bigcup_{1\le i\le n-1}\Gamma_i^\partial\right)\sqcup \Gamma_{n-1}^{\mathrm{aux}},
$$
and satisfies \textup{(i)}, \textup{(ii)} after given the colors of the $f_i$-edges adjacent to $A$. Fix such a coloring. Assume the $f_n$-edges $e=\{x,f_n(x)\}$ that are adjacent to $A$ are properly colored with colors in $\Gamma_n^\partial=\{c_n^j:1\le j\le k\}$. It remains to extend this coloring to the $f_n$-edges that are  contained in $A$.

We partition these edges according to their position relative to the backward faces
$$
F_1^-(A),\dots,F_{n-1}^-(A).
$$
Define
$$
E_{n-1}:=\{e\subseteq A: e \text{ is an }f_n\text{-edge and } e\cap F_{n-1}^-(A)=\varnothing\};
$$
for $1\le m\le n-2$, define
$$
E_m:=\left\{e\subseteq A:
e \text{ is an }f_n\text{-edge and }
e\subseteq \left(\bigcap_{i=m+1}^{n-1}F_i^-(A)\right)\setminus F_m^-(A)
\right\};
$$
and let
$$
E_0:=\left\{e\subseteq A:
e \text{ is an }f_n\text{-edge and }
e\subseteq \bigcap_{i=1}^{n-1}F_i^-(A)
\right\}.
$$
% Note that if $i\neq j$ and $e=\{x_1,x_2\}$ is an $f_{j}$-edge contained in $A$, then for any $y_1,y_2$ satisfying
% $$
% f_i(y_1)=x_1,\qquad f_i(y_2)=x_2,
% $$
% we have $y_1\in A\iff y_2\in A$.
% Therefore, for every $f_{j}$-edge $e\subseteq A$,
% $$
% e\subseteq F_i^-(A)\iff e\cap F_i^-(A)\neq\varnothing.
% $$
% It follows that every $f_n$-edge contained in $A$ belongs to exactly one of the sets $E_0,E_1,\dots,E_{n-1}$.

%更弱的版本↓
Note that if $i\neq j$ and $\{x,f_j(x)\}\subseteq A$. Choose
$\alpha_1,\ldots,\alpha_n,\beta_1,\ldots,\beta_n\in\mathbb N$
such that
$$
f_1^{\alpha_1}\cdots f_n^{\alpha_n}(x)=\operatorname{root}(A)
\quad\text{and}\quad
f_1^{\beta_1}\cdots f_n^{\beta_n}(f_j(x))=\operatorname{root}(A).
$$
Hence
$\alpha_i=\beta_i$ because $i\neq j$. By the definition of a
rooted region,
$$
x\in F_i^-(A)\iff \alpha_i=d_i(A),
\qquad
f_j(x)\in F_i^-(A)\iff \beta_i=d_i(A).
$$
Therefore,
$$
x\in F_i^-(A)
\quad\Longleftrightarrow\quad
f_j(x)\in F_i^-(A).
$$
Consequently, for every internal $f_j$-edge $e\subseteq A$,
$$
e\cap F_i^-(A)\neq\varnothing
\quad\Longleftrightarrow\quad
e\subseteq F_i^-(A).
$$
It follows that every $f_n$-edge contained in $A$ belongs to
exactly one of the sets $E_0,E_1,\ldots,E_{n-1}$.
%%结束

We now color these sets successively, starting from $E_{n-1}$ and proceeding down to $E_0$.

For $E_{n-1}$, let
$$
\Pi_{n-1}:=\{c_{n-1}^1,\dots,c_{n-1}^k,c_{n-1}^{k+1},c_{n-1}^{k+2},c_{n-1}^{k+3}\}.
$$
Fix a Borel linear ordering $\prec$ of $X$. We color the edges in $E_{n-1}$ recursively along the rooted $f_n$-trees, starting from the vertices in $F_n^+(A)$ and moving downward via $f_n^{-1}$.

First we make an observation. Let $e=\{x,y\}\in E_{n-1}$. By property \textup{(ii)} of the inductive coloring, those  already colored edges adjacent to $e$ whose colors lie in $\Pi_{n-1}$ are $f_{n-1}$-edges adjacent to $A$, since their colors are in 
$$
\Gamma_{n-1}^\partial=\{c_{n-1}^1,\dots,c_{n-1}^k\}.
$$
Since $e\cap F_{n-1}^-(A)=\varnothing$, each endpoint of $e$ is incident to at most one $f_{n-1}$-edge adjacent to $A$. Hence, at most two colors in $\Pi_{n-1}$ are forbidden, so some colors in $\Pi_{n-1}$ remain available.

Now suppose $|f_n^{-1}(y)|=q\le k,f_n(x_1)=\cdots=f_n(x_q)=y\in F_n^+(A)$, and $e_1=\{x_1, y\}, \dots, e_q=\{x_q, y\}\in E_{n-1}$, $x_1\prec x_2\prec \cdots \prec x_q$. We use a greedy algorithm to successively color $e_1, \dots, e_q$. By the above observation, at each step, there are at most $k+1$ many forbidden colors in $\Pi_{n-1}$ to avoid. Since $|\Pi_{n-1}|\geq k+2$, the algorithm works. This finishes the coloring of the $f_n$-edges one of whose endpoints is the top level of the $f_n$-tree. 

As we move downward via $f_n^{-1}$, suppose $e_1=\{x_1,y\}, \dots, e_q=\{x_q, y\}\in E_{n-1}$, $x_1\prec x_2\prec \cdots \prec x_q$, $f_n(x_1)=\cdots=f_n(x_q)=y$, and $\{y, f_n(y)\}\in A$ has been colored. We again use a greedy algorithm to successively color $e_1, \dots, e_q$. Now at each step there are at most $k+2$ many forbidden colors in $\Pi_{n-1}$ to avoid. Since $|\Pi_{n-1}|=k+3$, this algorithm still works. Thus we have colored all edges in $E_{n-1}$. 

For each $1\le m\le n-2$, let
$$
\Pi_m
:=
\Gamma_m^\partial
\cup
\bigcup_{j=m}^{n-2}\{c_j^{k+1},c_j^{k+2}\}
\cup
\{c_{n-1}^{k+1},c_{n-1}^{k+2},c_{n-1}^{k+3}\}.
$$
After the higher layers
$E_{m+1},\dots,E_{n-1}$
have already been colored, we now color edges in $E_m$ by colors in $\Pi_m$. 
% Note that for each such edge $\{x,f_n(x)\}$, edges adjacent to it with color in $\Pi_m$ must be of the form $\{x,f_i(x)\}$ or $\{f_n(x),f_if_n(x)\}$ for some $m\le i\le n-1$ or is the parent $f_n$-edge. Hence the same greedy counting argument applies, and $E_m$ can be colored properly using colors from $\Pi_m$.
Let $e=\{x,f_n(x)\}\in E_m$. Among the previously
colored $f_1,\ldots,f_{n-1}$-edges adjacent to
$e$, those whose colors lie in $\Pi_m$ can only be positive
$f_i$-edges for $m\le i\le n-1$. Hence at most
$2(n-m)$ colors in $\Pi_m$ are forbidden by such edges. In addition, the parent $f_n$-edge and the previously colored $f_n$-siblings forbid at most $1$ and $k-1$ additional colors, respectively.
Thus at most
$
2(n-m)+1+(k-1)=k+2(n-m)
$
colors are forbidden. Since
$
|\Pi_m|=k+2(n-m)+1,
$
at least one color remains available, and the greedy algorithm
properly colors $E_m$.

% Let
% $e=\{x,f_n(x)\}\in E_0$.
% For edges in $E_0$, we use colors in $\Gamma_n^{\mathrm{aux}}$ to color them. It easy to see that for $\{x,f_n(x)\}\in E_0$, we only need to be cautious about edges in the positive directions and its siblings. That is, there is at most $2n+k-1$ colors in $\Gamma_n^{\mathrm{aux}}$ that are adjacent to it. Thus the similar algorithm yields a proper coloring of $E_0$.

For an edge $e=\{x,f_n(x)\}\in E_0$, the only previously colored
edges that may use colors from $\Gamma_n^{\mathrm{aux}}$ are the
positive $f_i$-edges adjacent to its endpoints for
$1\le i\le n-1$, the parent $f_n$-edge, and the previously colored
$f_n$-siblings. Hence at most
$
2(n-1)+1+(k-1)=2n+k-2
$
colors are forbidden. Since
$|\Gamma_n^{\mathrm{aux}}|=3(n-1)+k+1$, the same greedy algorithm
properly colors $E_0$.
%结束

Combining the inductive coloring on the $f_1,\dots,f_{n-1}$-edges with the above coloring of the $f_n$-edges, we obtain a proper edge coloring of all edges meeting $A$. Clearly Property \textup{(i)} holds by the construction, and property \textup{(ii)} follows from the fact that a color in $\Gamma_i^\partial$ is used only on $f_i$-edges adjacent to $A$ or on $f_{j}$-edges contained in $A$ for some $j>i$.

Finally, since $A$ is finite, the set of all edges meeting $A$ is finite as well. Hence the resulting coloring is Borel.
\end{proof}

\begin{proof}[Proof of Theorem~\ref{thm:edgecoloring}]
 By Theorem~\ref{thm:Weilacher}, the hypothesis of Theorem~\ref{thm:rooted-markers} is satisfied.  Choose $d>0$, and let
$R_d\subseteq E_{f_1,\ldots,f_n}\upharpoonright F(X)$
be the rooted marker decomposition given by Theorem~\ref{thm:rooted-markers}.
Let $\Gamma_n$ be the color set from Lemma~\ref{lem:localcolor}.

Fix a Borel linear ordering $\prec$ of $F(X)$. We first color all edges crossing
distinct $R_d$-classes. If $e=\{x,f_i(x)\}$ is such an edge, assign to it a
color in $\Gamma_i^\partial$ according to the position of $x$ in the $\prec$-ordering of the finite
fiber $f_i^{-1}(f_i(x))$. This gives a Borel proper coloring $c_0$.

For each $R_d$-class $A$, the boundary coloring gives a proper coloring of
the edges adjacent to $A$. By Lemma~\ref{lem:localcolor}, it extends to a
proper coloring
$c_A:\{e:e\cap A\neq\varnothing\}\to\Gamma_n$.

Now define the global coloring $c$. If $e$ crosses two distinct
$R_d$-classes, let $c(e)$ be its boundary color $c_0(e)$. If $e\subseteq A$ for
some $R_d$-class $A$, set
$c(e)=c_A(e)$.
This is well-defined because internal edges belong to a unique $R_d$-class,
and crossing edges were colored globally before applying the local lemma.

The coloring is Borel and proper by the construction and
Lemma~\ref{lem:localcolor}. Hence
$$
\chi'_B\bigl(G_{f_1,\ldots,f_n}\upharpoonright F(X)\bigr)
\le |\Gamma_n|
=(n+1)(k+3)-5.
$$
\end{proof}

\addcontentsline{toc}{section}{References}
\bibliographystyle{amsalpha}
\bibliography{references}

\end{document}